\documentclass[onepage,12pt]{article}

\usepackage[utf8]{inputenc}
\usepackage{amsmath}
\usepackage{amsthm}
\usepackage{amssymb}
\usepackage{times}
\usepackage{bm}
\usepackage[numbers]{natbib}
\usepackage{amsfonts}
\usepackage{amssymb}
\usepackage{authblk}
\usepackage{mathtools}
\usepackage{dsfont}
\usepackage[margin=1in]{geometry}
\usepackage{xcolor}
\usepackage{array}
\usepackage{hyperref}
\usepackage{booktabs}
\usepackage{caption}
\usepackage{color,soul}
\usepackage{subcaption}

\usepackage[inline]{enumitem}

\usepackage{graphicx}
\DeclarePairedDelimiterX{\norm}[1]{\lVert}{\rVert}{#1}
\DeclarePairedDelimiterX{\abs}[1]{\lvert}{\rvert}{#1}

\newcommand\restr[2]{{
  \left.\kern-\nulldelimiterspace 
  #1 
  \right|_{#2} 
  }}

\newcommand{\diff}{\mathrm{d}}

\newcommand{\bbR}{\mathbb{R}}
\newcommand{\bbC}{\mathbb{C}}
\newcommand{\bbE}{\mathbb{E}}
\newcommand{\bbN}{\mathbb{N}}

\newcommand{\calA}{\mathcal{A}}
\newcommand{\calAv}{\mathcal{A}_{v}}
\newcommand{\calH}{\mathcal{H}}

\newcommand{\calB}{\mathcal{B}}
\newcommand{\calX}{\mathcal{X}}
\newcommand{\calY}{\mathcal{Y}}
\newcommand{\calF}{\mathcal{F}}
\newcommand{\calP}{\mathcal{P}}

\newcommand{\Tr}{\text{Tr}}

\newtheorem{theorem}{Theorem}[section]

\newtheorem{lemma}{Lemma}[section]
\newtheorem{definition}{Definition}[section]
\newtheorem{corollary}{Corollary}[section]
\newtheorem{proposition}{Proposition}[section]

\newtheorem{remark}{Remark}[section]
\newtheorem{assumption}{Assumption}[section]
\newtheorem{example}{Example}[section]

\title{A Fourier representation of kernel Stein discrepancy with application to Goodness-of-Fit tests for measures on infinite dimensional Hilbert spaces}

\author[1]{George Wynne}
\author[2,3]{Miko\l aj J. Kasprzak}
\author[4,5]{Andrew B. Duncan}
\affil[1]{University of Bristol}
\affil[2]{University of Luxembourg}
\affil[3]{MIT}
\affil[4]{Imperial College London}
\affil[5]{The Alan Turing Institute}

\date{}

\begin{document}

\makeatletter\let\Title\@title\makeatother
\maketitle


\begin{abstract}
Kernel Stein discrepancy (KSD) is a widely used kernel-based measure of discrepancy between probability measures. It is often employed in the scenario where a user has a collection of samples from a candidate probability measure and wishes to compare them against a specified target probability measure. KSD has been employed in a range of settings including goodness-of-fit testing, parametric inference, MCMC output assessment and generative modelling. However, so far the method has been restricted to finite-dimensional data. We provide the first analysis of KSD in the generality of data lying in a separable Hilbert space, for example functional data. The main result is a novel Fourier representation of KSD obtained by combining the theory of measure equations with kernel methods. This allows us to prove that KSD can separate measures and thus is valid to use in practice. Additionally, our results improve the interpretability of KSD by decoupling the effect of the kernel and Stein operator. We demonstrate the efficacy of the proposed methodology by performing goodness-of-fit tests for various Gaussian and non-Gaussian functional models in a number of synthetic data experiments.
\end{abstract}

\section{Introduction} \label{sec:introduction}
The kernel Stein discrepancy (KSD) \citep{Chwialkowski2016,Liu2016} is a kernel-based discrepancy between probability measures.  It provides a convenient approach to measure the divergence between a set of samples and a target probability measure which might only be known up to a normalization constant. The construction of KSD combines the Stein identity \citep{Stein1972,Ley2017,Chen2011}, which provides a set of sufficient conditions for a random variable to be distributed according to a given probability measure, and reproducing kernel Hilbert space (RKHS) theory.  Through the combination of these two tools, KSD has become an effective and generally-applicable tool in computational statistics and machine learning. Applications range from assessing MCMC output quality \citep{Gorham2017}, post-processing of MCMC output \citep{South2022}, goodness-of-fit testing \citep{Chwialkowski2016,Liu2016}, variational  and amortized inference  \citep{Liu2016SVGD,Fisher2021}, generalised Bayesian inference \citep{Matsubara2021} and generative modelling \citep{Grathwohl2020}. For a recent survey see \citep{Anastasiou2021}.

The goal of this paper is to tackle two central challenges of KSD. The first issue relates to the \emph{applicability} of KSD. By this we are specifically referring to the fact that theoretical and practical investigation of KSD has largely focused on finite dimensional Euclidean data. Some other contexts have been investigated, for example \citet{Yang2018} investigated discrete domains and \citet{Barp2018} studied compact Riemannian manifolds. Despite these advances existing theory does not cover the application of KSD-type discrepancies to the infinite dimensional Hilbert space context which is central to applications in non-parametric statistical modelling \citep{Ghosal2017}, Bayesian inverse problems \citep{Stuart2010} and functional data analysis \citep{horvath2012inference}.  The first aim of this paper is to extend the applicability of KSD to the setting of probability measures on infinite dimensional Hilbert spaces, establishing conditions under which it is able to separate distinct probability measures. Establishing the validity of KSD in this new setting would permit many of the aforementioned applications of KSD to be readily extended to infinite dimensional contexts.  We note that the KSD methodology has already found numerical application in infinite dimensions in the context of Stein variational gradient descent \citep{Jia2021} but without the accompanying theory which this paper aims to provide. 

The second issue concerns the \emph{sensitivity} of the behaviour of KSD with respect to parameter choices such as the choice of kernel and Stein operator. Understanding this is crucial to permit better performance of KSD-based statistical procedures. Currently the formulation of KSD is quite complicated and intertwines multiple parameter choices, making it hard to isolate the effect of each one. Therefore our desire is to find a representation of KSD where the different parameter choices have an isolated effect. 

These two issues will be addressed by deriving a \emph{Fourier representation} of KSD. Such a representation addresses the first aim since it will then become possible to establish conditions under which the KSD separates probability measures over infinite dimensional Hilbert spaces, meaning it is zero if and only if the two probability measures are equal. This is an essential property needed before KSD can be used in practice. The second aim is handled since the resulting Fourier decomposition isolates the effect of the kernel on the KSD which then makes clear the impact of kernel choice and hyper-parameters. The Fourier representation is achieved by combining developments in infinite dimensions of Stein's method \citep{Barbour1990,Shih2011,Bourguin2020} with elliptic equations for measures on Hilbert spaces \citep{Bogachev1995regularity,Albeverio1999uniqueness}. When the reference measure is Gaussian, this representation recovers the Stein-Tikhomirov method \citep{Tikhomirov1981}, where a partial differential equation is used to characterise a Gaussian distribution, see Remark \ref{rem:Stein_Tik}.

Once the separating property is established for probability measures over infinite dimensional Hilbert spaces, KSD may be employed. A central example of infinite dimensional data is functional data, studied in the field of functional data analysis (FDA), where data samples are functions, for example  curves, surfaces or images. In this setting, it is natural to view the samples as realisations of a probability measure supported on an infinite dimensional Hilbert space such as $L^{2}([0,1])$ \citep{horvath2012inference,Hsing2015,ferraty2006nonparametric}. The study has matured from initial developments in the 1990s \citep{Ramsay2005} into a broad field with multiple different applications and directions, for a recent review see \citep{Wang2016}. Common statistical tasks within FDA include regression \citep{Kadri2016}, classification \citep{Rossi2006}, two-sample testing \citep{Wynne2022}, Gaussianity testing \citep{Henze2020} and goodness-of-fit testing \citep{Ditzhaus2018}.

The main issue associated with statistical treatment of functional data is the infinite dimensional nature of the data. Many classical statistical procedures do not readily generalise to infinite dimensions. For example, as there is no infinite dimensional generalisation of the Lebesgue measure \cite{hunt1992prevalence} there is no canonical measure with respect to which a density can be defined, precluding the use of density-based methods. In fact, not only is there no canonical base measure but even the assumption that there exists a measure which both the user chosen target measure and a given candidate measure are both absolutely continuous with respect to is often too strong. Indeed, in infinite dimensions, stringent conditions are required for two Gaussian measures to be non-singular, see e.g.  \citet{Bogachev1998}. 

This issue is often side-stepped through the  ``project first'' approach to FDA, where functional data is first projected onto a finite dimensional subspace, after which classical  statistical procedures can be employed. The  particular choice of projection can be a fixed set of basis elements or can be data-driven, for example using functional principal components. This approach has been employed in two-sample testing for arbitrary difference of measures \citep{Pomann2016}, two-sample testing for difference of covariance operators \citep{Panaretos2010} and goodness-of-fit testing \citep{Bugni2009,Ditzhaus2018,CuestaAlbertos2007}. A challenge with this approach is the choice of the projection. If the chosen projections fail to sufficiently capture the variability of the random functions being investigated then the resulting procedure may be ineffective. In addition, the projections themselves rarely yield closed form expressions, and estimating them can be computationally non-trivial, for example requiring expensive Monte Carlo simulations. 

This work adopts an alternative approach, by formulating a discrepancy, and an associated goodness-of-fit test, directly on the infinite dimensional space and hence offers a totally different statistical paradigm compared to standard methods in FDA. The target measures we shall study are absolutely continuous with respect to a base Gaussian measure, this of course includes the case of Gaussian measures themselves. We call such measures Gibbs measures. This is a wide class of measures and is of central interest to functional data analysis since Gaussian measures can naturally be identified with Gaussian processes \citep{Rasmussen2006} and Gibbs measures can correspond to conditioned diffusions and solutions of stochastic differential equations \citep{Beskos2008}, see Section \ref{sec:numerics}. An application of KSD to goodness-of-fit testing of Gibbs measures is explored and in the Gaussian case compared to existing methods in FDA, for the Gibbs case we are not aware of any one-sample goodness-of-fit tests to compare to.

In summary, the contributions of this paper are
\begin{itemize}
    \item The formulation of KSD for probability measures on a separable Hilbert space.
    \item The derivation of a Fourier representation of KSD which provides new insight on the behaviour of KSD on separable Hilbert spaces, in both the finite and infinite dimensional settings.
    \item The identification of conditions which ensure that KSD can separate measures over separable Hilbert spaces and thus lead to consistent statistical procedures.
    \item The formulation of a one-sample goodness-of-fit test for Gibbs measures over separable Hilbert spaces.
    \item Demonstration of numerical performance of KSD based goodness-of-fit tests compared to existing approaches.
\end{itemize}

The rest of the paper is structured as follows, Section \ref{sec:KSD_intro} contains preliminary results and concepts required for the technical content of the paper. Section \ref{sec:KSD_for_GPs} introduces Stein operators and kernel Stein discrepancy and shows how KSD may be written in an easily estimated form. Section \ref{sec:KSD_Tik} contains the main contributions of the paper, showing how KSD can be written in a novel Fourier form which facilitates a proof of conditions sufficient for KSD to separate measures over infinite dimensional spaces. Section \ref{sec:numerics} contains synthetic numerical experiments, evaluating the performance of KSD as a basis for functional goodness-of-fit tests for Gaussian and non-Gaussian targets and as an evaluation tool to measure quality of simulations of paths of stochastic differential equations. Section \ref{sec:conclusion} contains concluding remarks. Proofs are in the supplement \citep{Wynne2023Supp}.

\section{Preliminaries}\label{sec:KSD_intro}

In this section the construction of Stein discrepancy and kernel Stein discrepancy \citep{Chwialkowski2016,Liu2016} is recalled. Given a topological space $\calX$, let $\calB(\calX)$ be the set of Borel measures on $\calX$ and $\calP(\calX)$ the set of Borel probability measures on $\calX$. The expectation of a measurable function $f$ of a random variable $X$ with law $P\in\calP(\calX)$ is denoted by $\bbE_{P}[f(X)]$ and when there are two independently, identically distributed versions of $X$ with respect to which expectation is taken, use $X,X'$ to denote the two random variables. For two Hilbert spaces $\calX,\calY$ denote by $L(\calX,\calY)$ the space of bounded linear maps from $\calX$ to $\calY$ and set $L(\calX) \coloneqq L(\calX,\calX)$.

\subsection{Stein's lemma and Stein discrepancies}\label{subsec:SD_intro}
Given a target Borel probability measure $P\in\calP(\calX)$ and a candidate Borel probability measure $Q\in\calP(\calX)$, the goal of a statistical discrepancy is to quantify how different $Q$ is from $P$. Integral probability metrics (IPMs) \cite{muller1997integral}  are a class of discrepancies which take the form
$D(Q,P) = \sup_{g\in\mathcal{G}}\left\lvert\bbE_{P}[g(X)] - \bbE_{Q}[g(X)]\right\rvert$ where $\mathcal{G}$ is a set of Borel measurable functions from $\calX$ to $\mathbb{R}$.  This defines a pseudo-metric on the space of probability measures and becomes a metric if $\mathcal{G}$ is sufficiently rich. Examples of IPMs include the Total Variation, Kantorovich and Dudley metrics \citep{Sriperumbudur2010}. By exploiting Stein's lemma \citep{Stein1972},  a family $\mathcal{G}$ can be constructed for which the resulting IPM involves expectations with respect to $P$ that can be computed trivially.  This is helpful since, in many applications of interest,  $Q$ will be an empirical measure, while computing expectations with respect to $P$ will be intractable.  To this end, given $P\in\calP(\calX)$, an operator $\calA$ and a set of functions $\calF$, lying in the domain of $\calA$, is called a Stein operator of $P$ and  a Stein class of $P$, respectively, if for every $Q\in\calP(\calX)$
\begin{align*}
    Q=P \iff \bbE_{Q}[\calA f(X)]=0 \quad \forall f\in\calF.
\end{align*}
The exact domain and range of the operator will be discussed when specific examples are employed.  The Stein discrepancy (SD) is obtained by choosing $\mathcal{G} = \mathcal{A}\mathcal{F}$ in the definition of an IPM
\begin{align}
    \text{SD}(Q,P) = \sup_{f\in\calF}\left\lvert\bbE_{P}[\calA f(X)] - \bbE_{Q}[\calA f(X)]\right\rvert = \sup_{f\in\calF}\left\lvert\bbE_{Q}[\calA f(X)]\right\rvert.\label{eq:SD}
\end{align}

A common technique to obtain Stein operators and Stein classes is the generator method. Namely, a Stein operator for a probability measure $P$ can be constructed from the infinitesimal generator of any Markov process which has unique invariant distribution $P$ \citep{Barbour1988,Barbour1990,Gotze1991}. This method is widely used due to the availability of Markov processes with closed form generators that are mathematically well-understood. 

\begin{example}[Langevin-Stein Operator]\label{exp:Afd}
For $\calX = \bbR^{d}$ and a measure $P\in\calP(\calX)$ with positive, differentiable density $p$ consider the It\^o stochastic differential equation on $\calX$ 
\begin{equation}\label{eq:langevin_sde}
\diff X_t = \nabla \log p(X_t)\diff t + \sqrt{2}\diff B_t,
\end{equation}
where $B_{t}$ is a standard Brownian motion. This is known as the overdamped Langevin equation. Under basic conditions on $p$ \citep{Gorham2017} the generator may be written
\begin{align}
    \calA f(x) = \Delta f(x) + \langle\nabla\log p(x),\nabla f(x)\rangle_{\bbR^{d}}.\label{eq:A_fd}
\end{align}
This can be used as a Stein operator and therefore is often called the Langevin-Stein operator. It is a popular choice since knowledge of the normalisation constant of $p$ is not required to compute $\calA$, making it appropriate, for example, in the setting where $p$ is a Bayesian posterior distribution \citep{Matsubara2021}.
\end{example}

The operator \eqref{eq:A_fd} is often used within the probability literature \citep{Nourdin2009,Anastasiou2021} due to its links to Markov processes. However, the evaluation of \eqref{eq:A_fd} involves taking second derivatives of $f$ which for convenience and computation purposes can be undesirable. Therefore in the machine learning and computational statistics literature it is common for a \emph{vectorisation} of the operator to be used instead. This is where the $\nabla f$ terms are replaced with a function $F\colon \calX\rightarrow\calX$ to reduce the number of derivatives involved in evaluating the operator. 

\begin{example}[Vectorised Langevin-Stein Operator]
    Continuing Example \ref{exp:Afd} with $\calX = \bbR^{d}$, if one starts from \eqref{eq:A_fd} and replaces $\nabla f$ with $F$ for some $F\colon\calX\rightarrow\calX$ then the result is the vectorised operator
    \begin{equation}\label{eq:vec_Afd}
        \calAv F(x) = \Tr[JF(x)] + \langle \nabla \log p(x),F(x)\rangle_{\bbR^{d}}.
    \end{equation}
    where $JF(x)$ is the Jacobian matrix of $F$ at $x$ and $\emph{Tr}$ is the matrix trace. This operator $\calAv$ is widely used in machine learning \citep{Gorham2015,Liu2016,Chwialkowski2016}. It is important to note that $\calAv$ is \emph{not} the generator of any Markov process since it acts on functions that take values in $\calX$.
\end{example}

Modifying Stein operators to suit the particular needs of a problem is a common approach in probability and statistics. Indeed, this is one of the strengths of Stein's method. Examples of these include the method of standardization \citep{Mijoule2018,Xu2022}. Therefore the act of vectorising a Stein operator should be seen as a modification to suit the purposes of the task at hand, namely vectorisation will offer an easier way to compute test statistics due to involving less derivatives.

This section has detailed how one can obtain Stein operators from Markov processes via the generator approach and how it is common to vectorise operators to make them easier to implement and compute. We believe it is important to investigate both non-vectorised and vectorised operators since the former is often studied in the probability literature and the latter in machine learning and computational statistics literature, therefore it is rare to see them compared and their properties contrasted. In Section \ref{sec:KSD_for_GPs} this strategy is adopted for measures on infinite dimensional Hilbert spaces, making use of infinitesimal generators arising from a \emph{gradient system} \citep{DaPrato2002,DaPrato2006} that is the infinite dimensional analogue of the overdamped Langevin equation. A vectorised version is then studied which simplifies some calculations due to the operator containing less derivatives and eases the implementation of the algorithm.

\subsection{Kernels and reproducing kernel Hilbert spaces}\label{subsec:KSD_intro}
Even though the Stein discrepancy \eqref{eq:SD}  cirmumvents the need to evaluate expectations with respect to $P$ the expression still requires evaluating the supremum over an infinite set of functions in the Stein class. The approaches presented in \citet{Chwialkowski2016,Liu2016} overcome this issue by choosing $\calF$ to be the unit ball of a reproducing kernel Hilbert space (RKHS) which is a Hilbert space of functions with special properties.  A Hilbert space $\calH$ of functions from $\calX$ to $\bbR$ is called a reproducing kernel Hilbert space \citep{Aronszajn1950} if there exists a function $k\colon\calX\times\calX\rightarrow\bbR$, called a kernel, that is symmetric and positive definite such that  (i) $k(\cdot,x)\in\calH, \:\forall x\in\calX$ and (ii) $\langle f,k(\cdot,x)\rangle_{\calH} = f(x),\:\forall f\in\calH,x\in\calX$. Property (ii) is called the \emph{reproducing property}. For each RKHS, the kernel $k$ is unique and for every kernel $k$ there exists an RKHS with $k$ as its kernel. Due to this one-to-one relationship we shall write $\calH =\calH_{k}$ and use $\langle\cdot,\cdot\rangle_{k}$ and $\norm{\cdot}_{k}$ to denote the inner product and norm, respectively. 

\begin{example}\label{exp:SE_IMQ}
	Let $\calX,\calY$ be Hilbert spaces and $T\in L(\calX,\calY)$. Then two examples of kernels are the Squared Exponential-$T$ (SE-$T$) and Inverse Multi Quadric-$T$ (IMQ-$T$) defined
	\begin{align*}
		k_{\emph{SE-}T}(x,y) & = e^{-\frac{1}{2}\norm{Tx-Ty}_{\calY}^{2}}\\
		k_{\emph{IMQ-}T}(x,y) & = (\norm{Tx-Ty}_{\calY}^{2}+1)^{-1/2}.
	\end{align*}
\end{example}
The notion of RKHS can be readily generalised from a space of functions mapping from $\calX$ to $\bbR$ to a space of functions mapping from $\calX$ to $\calX$ through the construction of operator-valued kernels \citep{Carmeli2010,Carmeli2006,Micchelli2005,Kadri2016}. A function $K\colon\calX\times\calX\rightarrow L(\calX)$ is an operator-valued kernel if $K(x,y) = K(y,x)^{*}\:\forall x,y\in\calX$ where $K(y,x)^{*}$ is the adjoint of $K(y,x)$, and if for every $n\in\bbN$ and $\{x_{i},y_{i}\}_{i=1}^{n}\subset\calX\times\calX$, the matrix $\left[\langle K(x_{i},x_{j})y_{i},y_{j}\rangle_{\calX}\right]_{i,j}$ is non-negative definite. The RKHS $\calH_{K}$ associated to $K$, is the unique Hilbert space of functions mapping from $\calX$ to $\calX$ which satisfies (i)\ $K(x,\cdot)y\in\calH_{K}\:\forall x,y\in\calX$ and (ii)\ $\langle f,K(x,\cdot)y\rangle_{K} = \langle f(x),y\rangle_{\calX}\:\forall f\in\calH_{K},x,y\in\calX$. 

An operator-valued kernel on $\calX$ can be easily constructed from a scalar kernel $k:\calX\times\calX \rightarrow \bbR$. Indeed, setting $K(x,y) = k(x,y)I_{\calX}$ where $I_{\calX}$ is the identity operator on $\calX$ satisfies the requirements indicated above. Given an operator-valued kernel $K$ of this form, let $f\in\calH_{K}$, the associated RKHS, and let $\{e_{i}\}_{i=1}^{\infty}$ be any orthonormal basis of $\calX$. Then the inner product on $\calH_{K}$ satisfies $\langle f,g\rangle_{K} = \sum_{i=1}^{\infty}\langle f_{i},g_{i}\rangle_{k}$ where $f = \sum_{i=1}^{\infty}e_{i}f_{i}, g = \sum_{i=1}^{\infty}e_{i}g_{i}$ with $f_{i},g_{i}\in\calH_{k}\:\forall i\in\bbN$ see \citet[Example 5]{Carmeli2010} or \citet[Example 6.5]{Paulsen2016}, and this is independent of the choice of basis. Therefore $\calH_{K}$ is a countable product of $\calH_{k}$ with norm $\norm{\cdot}_{K}$ given by $\norm{f}_{K}^{2} = \sum_{i=1}^{\infty}\norm{f_{i}}_{k}^{2}$.  

\subsection{Probability measures on Hilbert spaces}

Let $\calX$ be a separable Hilbert space, meaning a Hilbert space that contains a countable, dense subset, for example $L^{2}([0,1]^{d})$ for $d\in\bbN$. We are interested in probability measures on such spaces and the most common and easiest to use measures are Gaussian measures. Denote by $L^{+}_{1}(\calX)$ the space of symmetric, positive definite, trace class linear operators on $\calX$. Given $m\in\calX$ and $C\in L^{+}_{1}(\calX)$, the Gaussian measure with mean $m$ and covariance operator $C$, denoted $N_{m,C}$, is the unique probability measure on $\calX$ whose pushforward under the map $l(\cdot) = \langle y, \cdot\rangle_{\mathcal{X}}$ is Gaussian with mean $\langle m,y\rangle_{\calX}$ and variance $\langle Cy,y\rangle_{\calX}$, for all $y \in \calX$. For each Gaussian measure there exists a corresponding $m,C$ and for each $m\in\calX,C\in L^{+}_{1}(\calX)$ there exists a corresponding Gaussian measure. When $m=0$ we write $N_{m,C} = N_{C}$, for simplicity. Gaussian measures can be naturally identified with Gaussian processes \citep{Rasmussen2006,Rajput1972Lp}, which makes them highly interesting objects of study from the point of view of machine learning applications. For further details regarding Gaussian measures on Hilbert spaces see \citep{DaPrato2006,Maniglia2004,Bogachev1998}.

In this paper, we consider \textit{Gibbs} measures which is a class of measures strictly larger than just the Gaussian measures.

\begin{definition}\label{def:gibbs}
Call a measure $P\in\calP(\calX)$ a \emph{Gibbs measure} with respect to $N_{C}$ if the Radon-Nikodym derivative $\frac{dP}{dN_{C}}$ exists on $\calX$.  
\end{definition}
The main examples of Gibbs measures we shall consider in this work arise from solutions of stochastic differential equations (SDEs). Section \ref{sec:numerics} contains a worked example. The terminology Gibbs measure is used as it is standard in the stochastic partial differential equation literature from which many results are used, see \citet[Chapter 11]{DaPrato2006}.

\subsection{Derivatives of Hilbert space valued functions}\label{subsec:derivs}
Since our Stein operator will involve derivatives some elements of the theory of differentiation in Hilbert spaces must be introduced. In particular, one needs to first be introduced to the notion of a Fr\'echet derivative. 

Let $\calX$ and $\calY$ be Hilbert spaces.  Given $F\colon\calX\rightarrow \calY$, and $x \in \calX$, the Fr\'echet derivative (if it exists) of $F$ at $x$ is the function $DF\colon\calX\rightarrow L(\calX, \calY)$ satisfying
\begin{align*}
    	\lim_{\norm{h}_{\calX}\rightarrow 0} \frac{\norm{F(x+h) - F(x) -  DF(x)[h] }_{\calY}}{\norm{h}_{\calX}} = 0.
\end{align*}
If $\calY = \mathbb{R}$, identify $DF(x)$ with an element in $\calX$,  which we also denote by $DF(x)$, through the Riesz representation theorem
\begin{align*}
    \langle y, DF(x)\rangle_{\calX}= DF(x)[y]\:\forall x,y\in\calX.
\end{align*}
	 
Given a function $F\colon\calX\times \calX \rightarrow \mathbb{R}$ define $D_1 F(x, y)$ and $D_2 F(x,y)$ to be the partial Fr\'echet derivatives of $F$ at $x, y \in \calX$ with respect to the first and second variables, respectively, so $D_{1}F(x,y) = D(F(\cdot,y))(x)$ and $D_{2}F(x,y) = D(F(x,\cdot))(y)$. Again we identify them as elements of $\calX$ for each $x,y \in \calX$. Similarly, define the mixed partial Fr\'echet derivatives as $D_{2}D_{1}F(x,y) = D(D_{1}F(x,\cdot))(y)\in L(\calX)$ with analogous expressions for $D_{i}D_{j}F(x,y)\:i,j\in\{1,2\}$.

For $\alpha \in\bbN$ define $C_{b}^{(\alpha,\alpha)}(\calX\times\calX)$ as the space of functions $F\colon\calX\times\calX\rightarrow\bbR$ that have bounded, continuous partial Fr\'{e}chet derivatives up to order $\alpha$ on each argument. For example, if $\alpha = 1$ then $F\in C_{b}^{(1,1)}(\calX\times\calX)$ means $D_{1}^{i}D_{2}^{j}F$ is continuous and $\sup_{x,y\in\calX}\norm{D_{1}^{i}D_{2}^{j}F(x,y)}<\infty$ for $i,j\in\{0,1\}$, where the norm is appropriate for the order of derivative taken, so if $i = 1,j=0$ then the norm would be $\norm{\cdot}_{L(\calX,\bbR)}$.

For $C\in L^{+}_{1}(\calX)$ and a Hilbert space $\calY$ let $L^{2}(\calX,\calY;N_{C})$ denote the set of (equivalence classes of) functions $f\colon\calX\rightarrow\calY$ such that 
$\norm{f}_{L^{2}(\calX,\calY;N_{C})}^{2} = \int_{\calX}\norm{f(x)}_{\calY}^{2}dN_{C}(x) < \infty$. When $\calY = \bbR$ we will simply write $L^{2}(\calX;N_{C})$. Other $L^{p}$ spaces are defined similarly for other values of $p$.
The Malliavin-Sobolev space $W^{1,2}_{C}(\calX)$ is the subspace of $L^{2}(\calX;N_{C})$ such that the Malliavin derivative $C^{1/2}Df$ has finite $L^{2}(\calX,\calX;N_{C})$-norm.  Equipped with the inner product 
$$
\langle f, g\rangle_{W^{1,2}_C(\calX)} := \langle f, g\rangle_{L^{2}(\calX;N_{C})} + \langle C^{1/2}Df, C^{1/2}Dg\rangle_{L^{2}(\calX,\calX; N_{C})}$$ 
this defines a Hilbert space with Hilbertian norm $\lVert \cdot \rVert_{W^{1,2}_C(\calX)}$. This Malliavin-Sobolev inner product is analogous to the Sobolev inner product on $\bbR^{d}$ as it also involves $L^{2}$ inner products for the functions and their derivatives. It will afford us the ability to use integration by parts results critical to proving the main theorems much like how integration by parts results are typical for Sobolev spaces in finite dimensions. See \citet[Chapter 5]{Bogachev1998} for more discussion.

\section{Kernel Stein discrepancy for Gibbs measures}\label{sec:KSD_for_GPs}

In this section, a formulation of the kernel Stein discrepancy (KSD) for Gibbs measures on separable Hilbert spaces is presented,  generalizing the finite dimensional construction.   More specifically, given a target Gibbs measure $P$ and a candidate probability measure $Q$ on a separable Hilbert space $\calX$, we wish to define a discrepancy which does not require explicit knowledge of the normalisation constant of $P$, and only requires expectations with respect to $Q$.  The following assumptions shall be made on $\calX,P,Q$ before we formulate the discrepancy. 

\begin{assumption}\label{ass:X}
$\calX$ is a separable Hilbert space. 
\end{assumption}

\begin{remark}\label{rem:X_dim}
We focus on the case where $\calX$ is infinite dimensional in discussions and numerics since it is the case where novelty is provided. Also, for notational convenience all sums over the basis elements of $\calX$ are to infinity. To recover the finite dimensional version of the results simply replace the infinite sums with sums over each dimension of $\calX$. The requirement that the base space is a Hilbert space can preclude the use of some common compact spaces such as $[0,1]$, which is a limitation of our approach.
\end{remark}

\begin{assumption}
\label{ass:P_Q_assumptions}
	$P,Q\in\calP(\calX)$ and the target measure $P$ is a Gibbs measure, see Definition \ref{def:gibbs}, with $\frac{dP}{dN_{C}} \propto\exp(-U)$, such that
	\begin{equation}
	\label{eq:Q_conditions}
	\mathbb{E}_{Q}[\lVert X\rVert_{\calX}^{2}] , \mathbb{E}_{Q}[\lVert CD U(X)\rVert_{\calX} ] < \infty,
	\end{equation}
	and 
	\begin{equation}
	\label{eq:P_conditions}
		e^{-\frac{U(\cdot)}{2}} \in W^{1,2}_C(\calX), \mathbb{E}_{N_C}[\lVert C^{1/2}DU(X)\rVert_{X}^2] < \infty.
	\end{equation}
\end{assumption}

\begin{remark}
    The use of a Gaussian base measure for the target measure $P$ is essential to be able to use results in the literature related to infinite dimensional Langevin diffusions and measure equations \cite{Albeverio1999uniqueness,albeverio1991stochastic}. This is because the invariant measures of such diffusions have Gaussian base measures. Our results do not apply when the base measure is heavier tailed.
\end{remark}

The next assumptions impose regularity conditions on the kernel used in the formulation of the kernel Stein discrepancy. There are two versions of the assumptions, the first for when we employ a non-vectorised Stein operator and the second for when we employ a vectorised Stein operator. As discussed in Section \ref{subsec:SD_intro} the vectorisation process reduces the number of derivatives used in the Stein operator and operates on vector-valued functions, therefore the corresponding assumptions involve one less order of derivatives and an operator-valued kernel. 

\begin{assumption}
\label{ass:nonvec_k_ass}
The function $k\colon\calX\times\calX\rightarrow \bbR$ is a kernel satisfying $k\in C_{b}^{(2,2)}(\calX\times\calX)$.
\end{assumption}

\begin{assumption}\label{ass:vec_k_ass}
The function $k\colon\calX\times\calX\rightarrow \bbR$ is a kernel satisfying $k\in C_{b}^{(1,1)}(\calX\times\calX)$ and set $K(x,y) \coloneqq k(x,y)I_{\calX}$.
\end{assumption}

\begin{remark}
These two  assumptions on the regularity of $k$ are analogous to the regularity conditions imposed on $k$ in the finite dimensional case \citep{Gorham2017}. The requirement for two derivatives in Assumption \ref{ass:nonvec_k_ass} implies that the corresponding RKHS will have elements that are smoother than in the case of Assumption \ref{ass:vec_k_ass}. Later, in order to prove the novel Fourier representation of KSD, we will additionally require that the kernels be Fourier transforms of measures and hence translation invariant.
\end{remark}

\subsection{Formulating KSD using the generator method}

As discussed in Section \ref{subsec:SD_intro}, one of the most popular approaches to constructing Stein operators and Stein discrepancies is the generator method. This is the approach we take too. To this end, 
a stationary Markov process whose invariant measure is given by $P$ must be identified. The link between Stein's method and the construction of an associated Markov process has been well-known since \citet{Barbour1990}, where the construction of an Ornstein-Uhlenbeck process in a Banach space and associated infinitesimal generator is employed to quantify the error introduced by a functional diffusion approximation. In Section  \ref{subsec:SD_intro} it was noted that for the finite-dimensional context the overdamped Langevin diffusion described by the SDE \eqref{eq:langevin_sde} satisfies our requirements in the finite dimensional case. An analogous process exists that is Hilbert-valued and is appropriate for our purposes. This is shown in the following result taken from \citet[Theorem 4.1]{Bogachev1995regularity} and \citet[Theorem 6.10]{albeverio1991stochastic}.
\begin{proposition}\label{prop:diffusion}
    Suppose that Assumptions \ref{ass:X} and \ref{ass:P_Q_assumptions} hold for $Q=P\propto\exp(-U)N_C$, then there exists a $\calX$-valued Wiener process $B$ with covariance operator $C$ such that the stochastic differential equation
    \begin{equation}
    \label{eq:diffusion}
        \diff X_t = -\left(X_t + CDU(X_t)\right)\diff t  + \sqrt{2}\diff B_t,
    \end{equation}
    admits a weak solution $(X_t)_{t\geq 0}$ which is a $P$-symmetric diffusion process with invariant measure $P$.  
\end{proposition}

This diffusion is known as the pre-conditioned Langevin and has been studied in the context of sampling on function spaces \citep{Hairer2007}. The next step in the generator method is to identify the generator of the Markov process. This generator will then be used as our Stein operator. Define the operator $\calA$
\begin{align}
    \calA f(x)  = \text{Tr}(CD^{2}f(x)) - \langle Df(x),x+CDU(x)\rangle_{\calX}.\label{eq:grad_gen}
\end{align}
The domain of this operator has not yet been specified but the following result \citep[Remark 4.4]{Albeverio1999uniqueness} assures us that $\calA$ coincides with the generator of \eqref{eq:diffusion} on $\calF C_{b}^{\infty}(\calX)$,  defined 
\begin{align*}
\calF C_{b}^{\infty}(\calX) = \lbrace f \, | \, f(\cdot) = \phi(\langle l_1,\cdot\rangle_{\calX}, \ldots,\langle l_n, \cdot\rangle_{\calX}), \:l_1, \ldots, l_n \in \calX, \phi\in C_b^{\infty}(\mathbb{R}^n), n \in \mathbb{N}\rbrace, 
\end{align*}
where $C_{b}^{\infty}(\bbR^{n})$ is the space of bounded, infinitely differentiable functions from $\bbR^{n}$ to $\bbR$. 
\begin{proposition}\label{prop:generator}
   Under Assumptions \ref{ass:X} and \ref{ass:P_Q_assumptions} the space $\calF C_b^{\infty}(\calX)$ lies within the domain of the generator of \eqref{eq:diffusion} and the generator takes the form \eqref{eq:grad_gen} on $\calF C_b^{\infty}(\calX)$. 
\end{proposition}

The reason that $\calF C_{b}^{\infty}(\calX)$ is used to describe the action of the operator is that it is a large function space, being dense in many other spaces of interest. For the rest of this paper we use the form of the generator on $\calF C_{b}^{\infty}(\calX)$ and apply it to elements of an RKHS to construct the discrepancy. 

The operator \eqref{eq:grad_gen} is the generator of the diffusion \eqref{eq:diffusion}, which is an infinite dimensional, pre-conditioned analogue of the overdamped Langevin diffusion \eqref{eq:langevin_sde}. Therefore, it is natural to ask how the Langevin-Stein operator \eqref{eq:A_fd}, defined on $\bbR^{d}$, relates to \eqref{eq:grad_gen} when $\calX = \bbR^{d}$. The following example demonstrates the relation.

\begin{example}
Let $\calX = \bbR^{d}$ and $\calA_{L}$ be the Langevin-Stein operator \eqref{eq:A_fd} and $\calA$ the operator \eqref{eq:grad_gen} obtained as the generator of the pre-conditioned Langevin diffusion \eqref{eq:diffusion}. Let $P\in\calP(\calX)$ with a positive, differentiable density $p$ with respect to the Lebesgue measure. Let $C\in\bbR^{d\times d}$ be a covariance matrix and $N_{C}$ the be corresponding zero mean multivariate Gaussian on $\calX$. Then $\frac{dP}{dN_{C}}\propto\exp(-U)N_{C}$ with $U(x) = -\log p(x)-\frac{1}{2}\langle x,C^{-1}x\rangle_{\calX}$. As $U$ is now written in terms of $p$, the $x+CDU(x)$ term in $\calA$ can now be examined
\begin{align*}
    x+CDU(x) = x -C\nabla\log p(x) - C(C^{-1}x) = -C\nabla\log p(x).
\end{align*}
This shows that $-(x+CDU(x))$ plays the role of a pre-conditioned score function. When substituted into $\calA$ this gives $\calA f(x) = \Tr[CD^{2}f(x)]+\langle \nabla\log p(x), CDf(x)\rangle_{\calX}$ which is a pre-conditioned version of $\calA_{L}$, see \eqref{eq:A_fd}. This matches the analogy of how the initial diffusion for $\calA$ was a pre-conditioned version of the overdamped Langevin, the initial diffusion for $\calA_{L}$. In other words, the $x+CDU(x)$ term in $\calA$ represents the pre-conditioned score function of the density of $P$ with respect to the Lebesgue measure, $p$, whereas $\calA_{L}$ involves the non-preconditioned score function of $p$.
\end{example}

As outlined in Section \ref{sec:KSD_intro} it is helpful for theoretical and practical reasons to investigate a vectorised version of the operator, where the derivatives in \eqref{eq:grad_gen} are replaced with vector-valued functions, as is done in the statistical machine learning and computational statistics literature.

\begin{definition}[Stein operator]\label{def:non-vectorised}
Call the operator defined by
\begin{align}
\label{eq:A_nonvectorised_operator}
    \mathcal{A} f(x) = \Tr[C D^2 f(x)] - \langle Df(x) , x + C DU(x)\rangle_{\calX} \quad \forall x \in \calX, f \in \calH_k,
\end{align}
our Stein operator and the operator defined by 
\begin{align}
\label{eq:A_operator}
	\mathcal{A}_v F(x) = \Tr[C D F(x)] - \langle F(x) , x + C DU(x)\rangle_{\calX} \quad \forall x \in \calX, F \in \calH_K,
\end{align}
our vectorised Stein operator.
\end{definition}

The vectorisation of the operator replaced $D f$ with $F$ which shows that each component of the function $F$ is taking the place of each component of the derivative of $f$. However, $F$ does not have to be equal to the derivative of some function and hence the vectorisation provides a generalisation of the non-vectorised operator. We study both cases because the non-vectorised operator appears often in the probability literature and the vectorised one appears often in the computational statistics literature.

When $P$ is a Gaussian, meaning $U=0$, the operator $\calA$ has already been used to form a Stein discrepancy in infinite dimensions \citep{Shih2011,Bourguin2020}. Now armed with a Stein operator the kernel Stein discrepancy can be defined. Since the vectorised operator acts on functions that map from $\calX$ to $\calX$ the KSD that uses $\calAv$ involves $\calH_{K}$ instead of $\calH_{k}$.

\begin{definition}\label{def:ksd}
For a real-valued kernel $k$ on $\calX$ and for $K=kI_{\calX}$, the kernel Stein discrepancy (KSD) between probability measures $Q,P\in\calP(\calX)$ using $\calA$ and $\calAv$ are defined
\begin{align*}
     \text{KSD}_{\calA,k}(Q,P) & \coloneqq \sup_{\substack{f \in \calH_k, \\ \norm{f}_{k}\leq 1}}\left\lvert\bbE_{Q}[\calA f(X)]\right\rvert\\
    \text{KSD}_{\calA_v,K}(Q,P)& \coloneqq \sup_{\substack{F \in \calH_K, \\ \norm{F}_{K}\leq 1}}\left\lvert\bbE_{Q}[\calAv F(X)]\right\rvert,
\end{align*}
respectively.
\end{definition}

Since the operators $\mathcal{A}$ and $\mathcal{A}_v$ are now defined on the reproducing kernel Hilbert spaces $\calH_k$ and $\calH_K$, respectively, we need to ensure our assumptions result in well defined actions of $\calA$ and $\calAv$. 

\begin{lemma} \label{lemm:well_defined} 
Suppose Assumptions \ref{ass:X} and \ref{ass:P_Q_assumptions} hold. If $k$ satisfies Assumption \ref{ass:nonvec_k_ass} then \sloppy{$\emph{KSD}_{\calA,k}(Q,P)$} is well-defined and if $k$ satisfies Assumption \ref{ass:vec_k_ass} then $\emph{KSD}_{\calAv,K}(Q,P)$ is well-defined. 
\end{lemma}

It is important to note at this stage that $\text{KSD}_{\mathcal{A},k}(Q,P)$ or $\text{KSD}_{\mathcal{A}_v,K}(Q,P)$ might be zero when $Q\neq P$ meaning it is not a separating discrepancy. Deriving conditions on $k$ to ensure that the two KSD formulations can separate measures is a central part of KSD theory and is addressed in Section \ref{sec:KSD_Tik}.

\subsection{Expectation formulation of KSD}
For the rest of this section we show how KSD can be rewritten in a way which makes it easily estimated in practical statistical tasks. Before proceeding recall the discussion in Remark \ref{rem:X_dim} which states that though infinite sums are used for convenience, and to emphasise that the main contribution lies in the infinite dimensional case, the results also hold when $\calX$ is finite dimensional. To recover the finite dimensional case one simply changes the sums to be over the dimension of $\calX$. To this end, consider the following operator acting on Fr\'{e}chet differentiable functions $f\colon\calX\rightarrow\bbR$
\begin{equation}\label{eq:operator_gamma}
    \Gamma f(x)=CDf(x)-(x+CDU(x))f(x).
\end{equation}
Fix an orthonormal basis $\{e_i\}_{i=1}^{\infty}$ of $\mathcal{X}$ and set $F_i(x)=\left<F(x),e_i\right>_{\calX}$. Then 
\begin{align}
    \mathcal{A}_vF(x) = \sum_{i=1}^{\infty}\left<\Gamma F(x),e_i\right>_{\mathcal{X}} = \sum_{i=1}^{\infty}\Gamma_{i}F_{i}(x),\label{eq:gamma_trace}
\end{align}
where $\Gamma_{i}f(x)\coloneqq \langle\Gamma f(x),e_i\rangle_{\mathcal{X}}$ for $f\colon\calX\rightarrow\bbR$. The expression \eqref{eq:gamma_trace} shows that $\calAv$ behaves similar to a trace norm, the $i$-th component of $F$ is being projected into the $i$-th direction. This interpretation helps obtain the next result. It is interesting that a trace norm interpretation of the vectorised Stein operator was noted in one of the first KSD papers \citep{Liu2016}.

\begin{theorem}\label{thm:KSD_OU}
Suppose Assumptions \ref{ass:X} and \ref{ass:P_Q_assumptions} hold. If $k$ satisfies Assumption \ref{ass:nonvec_k_ass} then 
\begin{align}
    \emph{KSD}_{\mathcal{A},k}(Q,P)^2 =  \bbE_{(X,X') \sim Q\times Q}[(\calA\otimes\calA)k(X,X')],\label{eq:KSD_double_int1}
\end{align}
and if $k$ satisfies Assumption \ref{ass:vec_k_ass} then
\begin{align}
    \emph{KSD}_{\mathcal{A}_v,k}(Q,P)^2 = \sum_{i=1}^{\infty}\bbE_{(X,X') \sim Q\times Q}[(\Gamma_i\otimes\Gamma_i)k(X,X')]\label{eq:KSD_double_int}.
\end{align}
\end{theorem}

This shows KSD can be written as a double expectation with respect to the candidate measure $Q$ for both the standard and vectorised case. The proof is largely the same as in finite dimensions \citep{Gorham2017,Chwialkowski2016} except that it needs to be established that the kernel can reproduce derivatives in infinite dimensions, which is a technical contribution we provide. Theorem \ref{thm:KSD_OU} has important practical implications as it means that KSD can be estimated as a $U$-statistic using only samples from $Q$. This has been known already to be the case for finite dimensional $\calX$ and it is important that this property still holds for infinite dimensional $\calX$. Section \ref{sec:numerics} outlines how KSD is estimated using these double expectation expressions. The right-hand sides of \eqref{eq:KSD_double_int1} and \eqref{eq:KSD_double_int}  may be expanded to give the following corollary. The proof is contained in the proof of Theorem \ref{thm:KSD_OU}.

\begin{corollary}\label{cor:stein_kernels}
Suppose Assumptions \ref{ass:X} and \ref{ass:P_Q_assumptions} hold. If $k$ satisfies Assumption \ref{ass:nonvec_k_ass} then 
\begin{align}
     \emph{KSD}_{\mathcal{A},k}(Q,P)^2 =  \bbE_{(X,X') \sim Q\times Q}[h(X,X')],\label{eq:nonvec_double_int}
\end{align}
where 
\begin{align*}
 h(x,x') & = \sum_{i,j=1}^{\infty}\lambda_{i}\lambda_{j}D_{2}^{2}D_{1}^{2}k(x,x')[e_{i},e_{i},e_{j},e_{j}] \\
 & - \sum_{i=1}^{\infty}\lambda_{i}D_{2}^{2}D_{1}k(x,x')[x+CDU(x),e_{i},e_{i}]\\
   &  - \sum_{i=1}^{\infty}\lambda_{i}D_{1}^{2}D_{2}k(x,x')[x'+CDU(x'),e_{i},e_{i}] \\
   & + D_{2}D_{1}k(x,x')[x+CDU(x),x'+CDU(x')].
\end{align*}
If $k$ satisfies Assumption \ref{ass:vec_k_ass} then
\begin{align}
     \emph{KSD}_{\calAv,K}(Q,P)^2 =  \bbE_{(X,X') \sim Q\times Q}[h_{v}(X,X')],\label{eq:vec_double_int}
\end{align}
where
\begin{align*}
	h_{v}(x,x')  & =   \emph{Tr}[C^{2} D_{2}D_{1}k(x,x')] - \langle D_{1}k(x,x'), Cx'+C^{2}DU(x')\rangle_{\calX} \\
	&\quad - \langle D_{2}k(x,x'), Cx+C^{2}DU(x)\rangle_{\calX} + k(x,x')\langle x+CDU(x),x'+CDU(x')\rangle_{\calX}.
\end{align*}
\end{corollary}

These expressions are somewhat large so we now provide examples of kernels which satisfy our assumptions and the corresponding expressions for $h_{v}$. The expressions for $h$ we do not include for brevity but the proof of Proposition \ref{prop:ker_example} should instruct the reader how to derive them.  

\begin{proposition}\label{prop:ker_example}
	Under Assumption \ref{ass:X} let $T\in L(\calX)$, then the SE-$T$ and IMQ-$T$ kernels from Example \ref{exp:SE_IMQ} satisfy Assumption  \ref{ass:nonvec_k_ass} and Assumption \ref{ass:vec_k_ass} and their corresponding $h_{v}$ expressions from Corollary \ref{cor:stein_kernels} are
	\begin{align*}
	h_{v}^{\emph{SE}}(x,y) & = k_{\emph{SE-}T}(x,y)\big(\langle x+CDU(x),y+CDU(y)\rangle_{\calX} - \langle SC(x-y),x-y\rangle_{\calX} \\
	& - \langle SC(CDU(x)-CDU(y)),x-y\rangle_{\calX} + \emph{Tr}(SC^{2}) - \norm{CS(x-y)}_{\calX}^{2}\big)\\
	h_{v}^{\emph{IMQ}}(x,y) & = k_{\emph{IMQ-}T}(x,y)\langle x+CDU(X),y+CDU(y)\rangle_{\calX}\\
	& + k_{\emph{IMQ-}T}(x,y)^{3}\big(\emph{Tr}(SC^{2}) - \langle SC(x-y),x-y\rangle_{\calX} \\
	& - \langle SC(CDU(x)-CDU(y)),x-y\rangle_{\calX}\big) \\
    & -3k_{\emph{IMQ-}T}(x,y)^{5}\norm{CS(x-y)}_{\calX}^{2},
	\end{align*}
	where $S = T^{*}T$. 
\end{proposition}

\section{Fourier representation of kernel Stein discrepancy}\label{sec:KSD_Tik}
The KSD given by Definition \ref{def:ksd} will not necessarily separate measures in $\calP(\calX)$, meaning it is possible that $\mbox{KSD}(P, Q) = 0$ when $Q \neq P$. Conditions on the kernel and Stein operator to ensure that the KSD can separate measures have been given in multiple scenarios in the finite dimensional case \citep{Gorham2015,Chwialkowski2016,Liu2016} but all rely upon having probability density functions. This is of course not possible in the infinite dimensional case due to densities not existing. To circumvent this we establish a novel link between KSD and elliptic measure equations. This link occurs by extending 
the current popular methodology of the generator method to include studying the measure equation defined by the generator used as the Stein operator. 

Before this, some new notation must be introduced. Given a Borel measure $\nu \in \calB(\calX)$ let $\widehat{\nu}$ be the characteristic function, also known as the Fourier transform, of $\nu$ defined by $
\widehat{\nu}(s) = \int e^{i\langle s, x\rangle_{\calX}}d\nu(x)
$, for $s \in \calX$ and where $i$ is the imaginary unit.  Let  $\calX_{\bbC}$ denote the complexification of $\calX$, so that  ${\calX_{\bbC}\coloneqq\{a+ib\colon a,b\in\calX\}}$ with associated inner product 
$$\langle a+ib,c+id\rangle_{\calX_{\bbC}} = \langle a,c\rangle_{\calX}+\langle b,d\rangle_{\calX}+i\langle b,c\rangle_{\calX}-i\langle a,d\rangle_{\calX}.$$ 
This inner product is used in the proofs involving the vectorised operator to derive expressions which involve scalar multiplications of complex exponentials with elements of $\calX$. See \citet{Paulsen2016} for further discussion about complexification of Hilbert spaces. Now, recall the operator $\calA$ defined by \eqref{eq:grad_gen}.  This is used to form a measure equation \citep{Bogachev2009survey,Bogachev1995regularity,Albeverio1999uniqueness,Bogachev2010}. Specifically, following \citet{Bogachev1995regularity}, for a Borel measure $Q$ on $\calX$, write 
$$
    \calA^* Q = 0
$$
if the following two conditions are satisfied
\begin{equation}
\label{eq:U_condition}
    \langle y, \cdot + CDU(\cdot)\rangle_{\calX} \in L^2(\calX;Q)\quad \forall y \in \calX,
\end{equation}
and 
\begin{equation}
\label{eq:integral_condition}
    \bbE_{Q}[\calA f(X)] = 0 \quad\forall f\in \calF C_b^\infty(\calX).
\end{equation}

Extending $\calA$ to take complex-valued functions and using the fact that $\mbox{span}\lbrace e^{i\langle s, \cdot\rangle_{\calX}} \, :\, s \in \calX\rbrace \subset \calF C_{b}^{\infty}(\calX)$ is dense in $L^2(\calX; Q)$, the following result from  \citet[Remark 3.13]{Bogachev1995regularity} provides a Fourier condition which is equivalent to \eqref{eq:integral_condition}.

\begin{proposition}  
\label{prop:fourier_solution}
Suppose that Assumptions \ref{ass:X} and \ref{ass:P_Q_assumptions} hold then $Q$ satisfies \eqref{eq:integral_condition} if and only if 
\begin{equation}
    \bbE_{Q}[\calA (e^{i\langle s, \cdot\rangle_{\calX}})(X)] = 0 \quad \forall s \in \calX,
\end{equation}
which can be written as 
\begin{equation}
    \left\langle\mathbb{E}_Q\left[\Gamma(e^{i\langle s,\cdot\rangle_{\calX}})(X)\right],s\right\rangle_{\calX_{\bbC}}=0\quad \forall s\in\calX, \label{eq:gamma_crit}
\end{equation}
where $\Gamma$ is the operator defined in \eqref{eq:operator_gamma}.
\end{proposition}

The proof of this result in the case $\calX = \bbR^{d}$ relies on a simple integration by parts argument involving the score function $\nabla\log p(x)$ where $p$ is the density of the target with respect to the Lebesgue measure \citep[Proposition 1]{Gorham2015}. The difficulty in the current infinite dimensional case is two-fold. First the identification of an appropriate analogy to $\nabla\log p(x)$ and second the identification of an appropriate analogy to an integration by parts result. The first of these is tackled through the notion of a logarithmic derivative \citep{Bogachev2010}. The second is through the existence of integration by parts type results for logarithmic gradients. See Section \ref{subsec:deriv_measures} in the supplement \citep{Wynne2023Supp} for more discussion.

The reason for looking at the measure equation is that existence of solutions to  \eqref{eq:integral_condition} is, under certain assumptions, equivalent to the existence of invariant measures of the associated Markov process, given by \eqref{eq:diffusion}, and uniqueness can also be obtained. This is made concrete in the following result \citep[Theorem 4.5]{Albeverio1999uniqueness}.

\begin{proposition}
\label{prop:unique}
Suppose that Assumptions \ref{ass:X} and \ref{ass:P_Q_assumptions} hold for $Q = P \propto e^{-{U}}N_C$.  Then $P$ is the unique probability measure that satisfies \eqref{eq:U_condition} and \eqref{eq:integral_condition}.  
\end{proposition}

\begin{remark}
The proof for the above result largely rests on integration-by-parts type results which are discerned using logarithmic derivatives. This is the infinite dimensional analogue to integration by parts type results that use density functions in finite dimensions. The interested reader may consult \citet{Bogachev2009survey,Bogachev2010} and the supplement for more details on such derivatives.
\end{remark}

Currently it has been established that measure equations can provide a characterisation of our target measure and the criterion can be written in terms of $\Gamma$, see \eqref{eq:gamma_crit}, which is related to the Stein operator $\calA$. The next result shows a Fourier representation of KSD which relates KSD to the criterion \eqref{eq:gamma_crit} and hence to the measure equation. 

It is at this stage that our analysis focuses on kernels which are the Fourier transforms of measures, and hence are translation invariant. The requirement of this property is essential for what follows as it provides the connection between KSD and the novel Fourier representation. A kernel not having this property, for example a non-translation invariant kernel, would not possess such a representation of KSD. This disqualifies some kernels which have been used for functional data such as by \citep{Chevyrev2018,Salvi2021}.

\begin{theorem}\label{thm:Tik_KSD}
    Suppose Assumptions \ref{ass:X} and \ref{ass:P_Q_assumptions} and $k(x,y) = \widehat{\mu}(x-y)$ for some $\mu\in\calB(\calX)$. If $k$ satisfies Assumption \ref{ass:nonvec_k_ass} then
    \begin{align}
        \emph{KSD}_{\calA, k}(Q,P)^{2} = \int_{\calX}\left|\bbE_{Q}[\calA (e^{i\langle s, \cdot\rangle_{\calX}})(X)] \right|_{\mathbb{C}}^2 d\mu(s),\label{eq:KSD_nonvec_int}
    \end{align}
    and if $k$ satisfies Assumption \ref{ass:vec_k_ass} then
    \begin{align}
        &\emph{KSD}_{\calA_v, K}(Q,P)^{2} = \int_{\calX}\left\|\mathbb{E}_Q\left[\Gamma(e^{i\left<s,\cdot\right>_{\calX}})(X)\right]\right\|_{\calX_{\mathbb{C}}}^2 d\mu(s).\label{eq:KSD_vec_int}
    \end{align}
\end{theorem}

This representation of KSD is called a Fourier representation since for such $k$ the measure $\mu$ is known as the Fourier measure of $k$ and the KSD has been written in a form which only includes the action of the operators $\calA$ and $\Gamma$ on complex exponentials, weighted by the Fourier measure $\mu$.

\begin{remark}\label{rem:Stein_Tik}
In the special case where $P = N_C$, so that $U = 0$ and $\Gamma f(x) = CDf(x) - xf(x)$, the Fourier characterisation for the vectorised case \eqref{eq:KSD_vec_int} reduces to
$$
    \text{KSD}_{\calAv,K}(Q,N_C)^{2} = \int_{\calX}\left\|Cs\widehat{Q}(s) + D\widehat{Q}(s)\right\|_{\calX_{\bbC}}^{2}d\mu(s),
$$
which we recognize to be very similar to the test-statistic introduced in \citet{Ebner2020review}, which was employed for goodness-of-fit testing for finite dimensional data.   This test-statistic is based on the Stein-Tikhomirov method \citep{Tikhomirov1981,Arras2016}, leveraging the observation that $\phi = \widehat{N}_{C}$ is the unique solution to the differential equation $Cs\phi(s)+D\phi(s) = 0, \phi(0) = 1$ so the integrand quantifies the discrepancy between $Q$ and $P$ in terms of the magnitude of the residual induced by plugging $\widehat{Q}$ into this differential equation.  This interesting connection demonstrates how KSD subsumes this particular methodology as a special case, which to our knowledge, was not previously known.
\end{remark}

This Fourier representation provides insight into the behaviour of KSD, specifically the influence of the kernel $k$ on the disprepancy. Indeed, a key aspect of \eqref{eq:KSD_nonvec_int} and \eqref{eq:KSD_vec_int} is that the kernel choice only influences the integrating measure $\mu$ while the integrand is determined entirely by the choice of Stein operator. This is in contrast to the initial definition of KSD in Definition \ref{def:ksd} where the kernel choice and Stein operator choice interact in a more complicated fashion. Therefore the Fourier representation addresses the second main aim of this paper, to increase interpretability of kernel Stein discrepancy. 

Indeed, the heavier the tails of $\mu$, the more weight is placed upon test functions $\calA(e^{i\langle s,\cdot\rangle_{\calX}})(\cdot)$ for values of $s$ with large norm. This can result in the test functions being more complex and thus the KSD becomes more discerning between $Q$ and $P$ since the expectation with respect to $Q$ has to match the expectation with respect to $P$ for these more complex functions. The following example illustrates this in the case $\calX = \bbR$.

\begin{example}\label{exp:Test_Functions}
    Let $\calX = \bbR$, $P$ have density $p(x)\propto \exp(-\left(\frac{x-3}{3}\right)^{2})$ and $\calA$ be the standard Langevin-Stein operator from \eqref{eq:A_fd} meaning $\calA f(x) = f''(x) + (\log p)'(x)f'(x)$. In this case the test functions are
    $$
    \calA(e^{is\cdot})(x) = -s^{2}\cos(sx)-s\sin(sx)(\log p)'(x) + i\left(-s^{2}\sin(sx)+s\cos(sx)(\log p)'(x)\right).
    $$
    Clearly, larger values of $s$ result in the test function $\calA(e^{is\cdot})$ having higher periodicity, from the trigonometric terms, and amplitude, from the $s^{2}$ terms. Figure \ref{fig:Test_Functions} illustrates this by plotting the real part of $\calA(e^{is\cdot})(x)$ over the range $[-10,10]$ for $10$ i.i.d.\ samples of $s$ from different choices of $\mu$ hence different kernels. In Figure \ref{fig:Gauss_Test} where $\mu$ is a standard normal having light tails and corresponding to a squared exponential kernel, the samples of $s$ result in test functions of limited complexity. In Figure \ref{fig:Students_Test} where $\mu$ is a Students-$t$ distribution, corresponding to a Mat\'{e}rn kernel, the tails are heavier than Gaussian and some of the samples of $s$ result in test functions with notable higher magnitude and periodicity. Finally, in Figure \ref{fig:Cauchy_Test} where $\mu$ is a Cauchy distribution, corresponding to the Laplace kernel, the heavy tails result in highly erratic test functions. 
\end{example}

  Example \ref{exp:Test_Functions} shows heavier tailed $\mu$ results in more erratic test functions and hence a more discerning KSD. This observation is purely informal and it would be interesting, but beyond the scope of this paper, to leverage the Fourier representations in Theorem \ref{thm:Tik_KSD} to fully characterise conditions, purely in terms of the tails of $\mu$, for when KSD is discerning enough to metrise weak convergence.

\begin{figure}[ht]
  \begin{subfigure}{0.31\textwidth}
    \includegraphics[width=\linewidth]{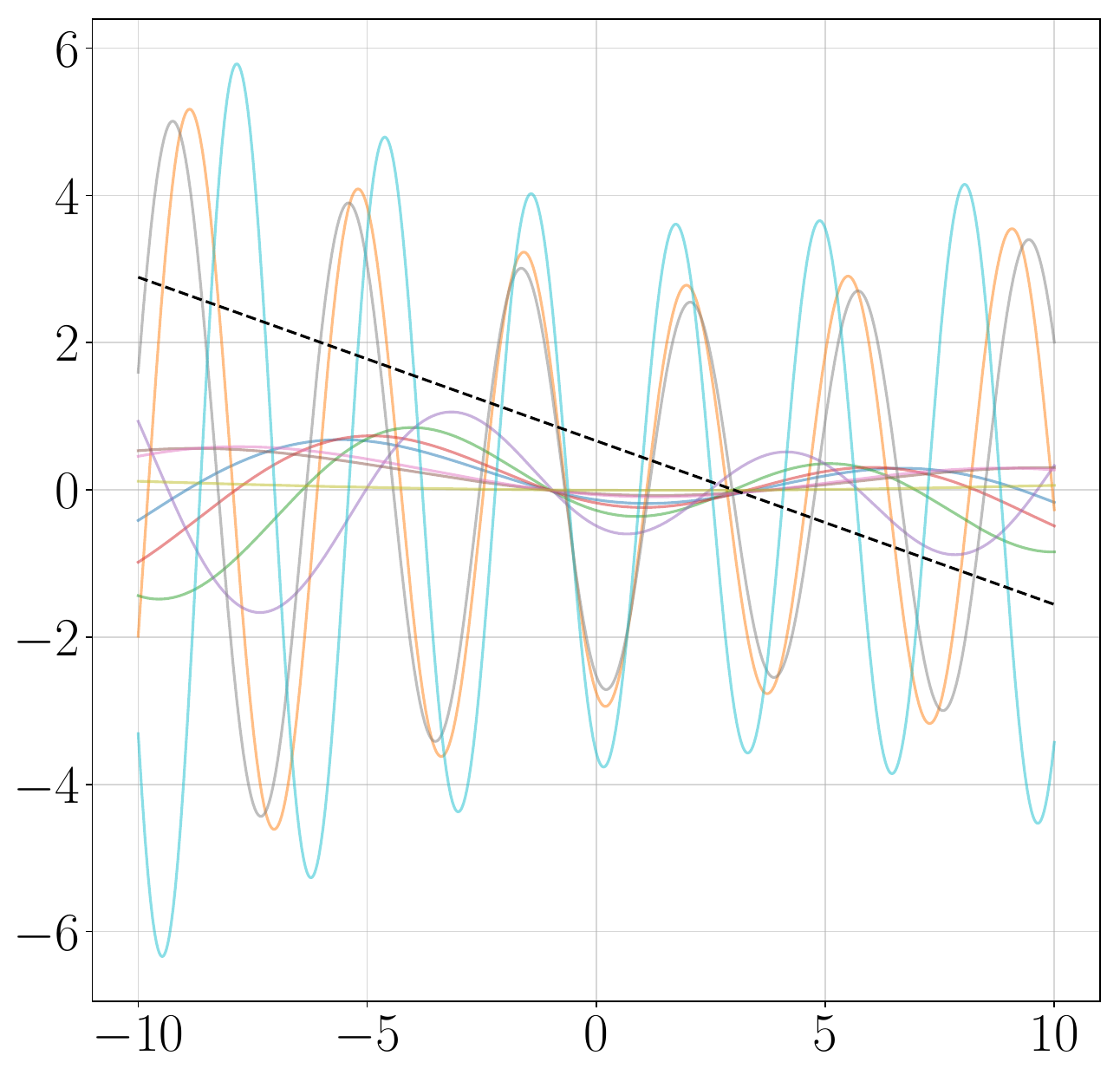}
    \caption{$\mu = \text{Gaussian}(0,1)$} \label{fig:Gauss_Test}
  \end{subfigure}%
  \hspace*{\fill} 
  \begin{subfigure}{0.31\textwidth}
    \includegraphics[width=\linewidth]{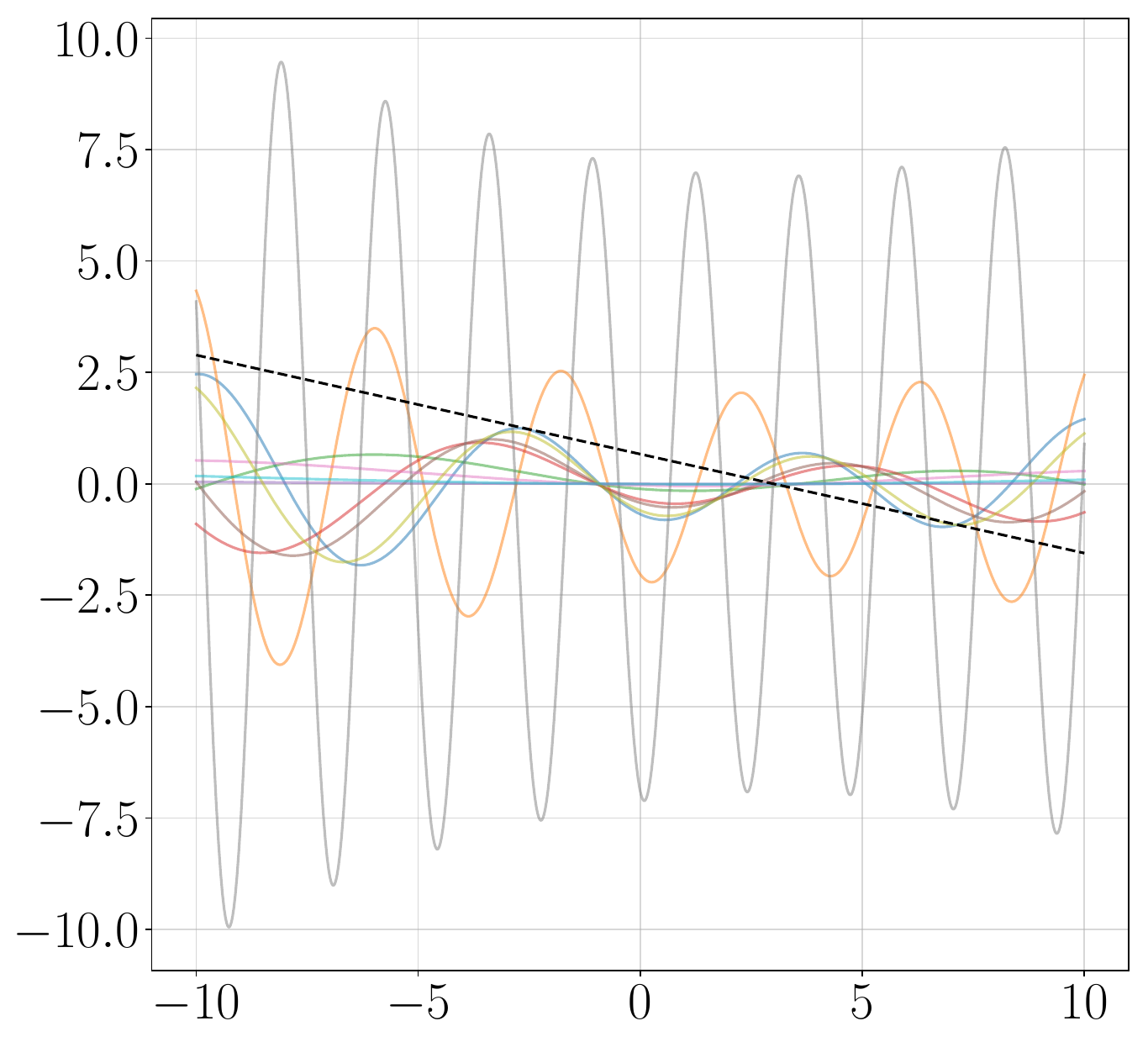}
    \caption{$\mu = \text{Students-}t(\nu=2)$} \label{fig:Students_Test}
  \end{subfigure}
  \hspace*{\fill} 
  \begin{subfigure}{0.31\textwidth}
    \includegraphics[width=\linewidth]{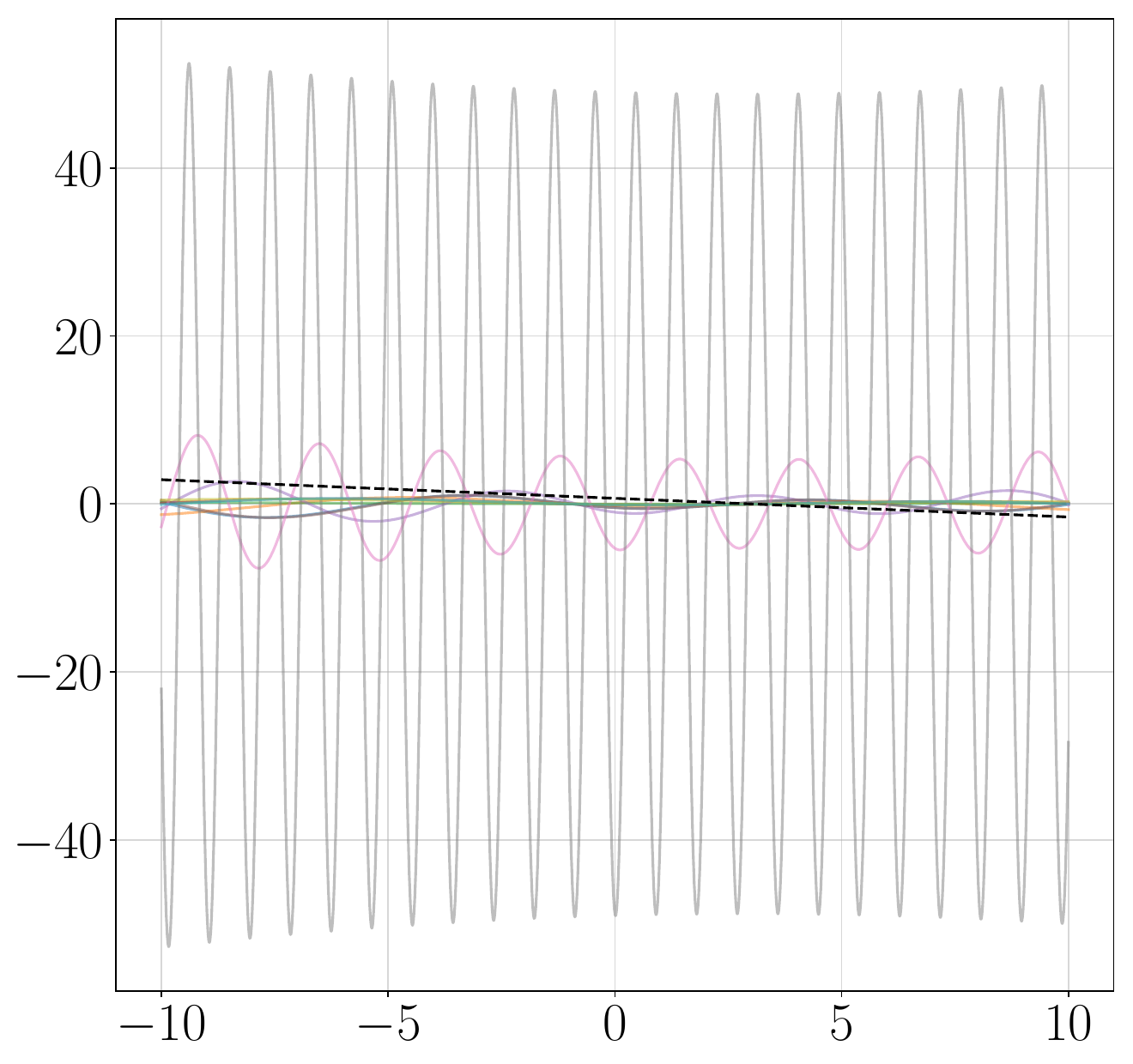}
    \caption{$\mu = \text{Cauchy}(0,1)$} \label{fig:Cauchy_Test}
  \end{subfigure}

\caption{Plots corresponding to Example \ref{exp:Test_Functions} of the real part of the test functions $\calA(e^{is\cdot})(x)$ for $10$ samples from different choices of $\mu$, the heavier the tails of $\mu$ the larger the samples of $s$ hence the greater the magnitude and periodicity of the test functions. In black is $(\log p)'(x)$ where $p(x)\propto \exp(-\left(\frac{x-3}{3}\right)^{2})$.} \label{fig:Test_Functions}
\end{figure}

More intuition can be gained when the finite dimensional case $\calX=\bbR^{d}$ is considered which is described in the next example. 

\begin{example}
	In the case where $\calX = \mathbb{R}^d$ and $C = I$, the argument of Theorem \ref{thm:Tik_KSD} can be repeated for KSD based on the vectorised Langevin-Stein operator \eqref{eq:vec_Afd}. Supposing $P,Q$ have differentiable densities $q,p$ with respect to the Lebesgue measure
	\begin{align}
		\text{KSD}_{\calAv,K}(Q,P)^{2} = \int_{\bbR^{d}}\norm[\bigg]{s\widehat{q}(s) - i\int_{\bbR^{d}}\nabla\log p(x)e^{i\langle s,x\rangle_{\bbR^{d}}}q(x)dx}_{{\bbC}^d}^{2}d\mu(s),\label{eq:spec_fin}
	\end{align}
	where $\widehat{q}(s) \coloneqq \int_{\bbR^{d}}q(x)e^{i\langle s,x\rangle_{\bbR^{d}}}dx$. Note that
	\begin{align*}
		i\int_{\bbR^{d}}\nabla\log p(x)e^{i\langle s,x\rangle_{\bbR^{d}}}p(x)dx = i\int_{\bbR^{d}}\nabla p(x)e^{i\langle s,x\rangle_{\bbR^{d}}}dx = i\widehat{\nabla p}(s) = s\widehat{p}(s),
	\end{align*} 
	 where the classical Fourier derivative identity $ i\widehat{\nabla p}(s) = s\widehat{p}(s)$ has been used for the final step. Therefore, the Fourier representation reveals the KSD is measuring how much this Fourier derivative identity is being violated when some of the $p$ terms are replaced with $q$, weighted according to the Fourier measure of the kernel.
\end{example}

The central assumption which underpins the Fourier characterisation of KSD is that the kernel $k$ can be expressed as the Fourier transform of some Borel measure $\mu \in \calB(\calX)$, meaning $k(x,y)  = \widehat{\mu}(x - y)$.  When $\calX$ is finite dimensional, Bochner's theorem \citep[Theorem 4.4]{Kukush2019} states that $k$ being translation invariant and continuous is necessary and sufficient.  However, this does not hold when $\calX$ is infinite dimensional.  Indeed, by the Minlos-Sazonov theorem \citep[Theorem 4.5]{Kukush2019}, it is necessary and sufficient that the kernel is translation invariant and continuous with respect to the Sazonov topology \citep[Section 4.4]{Kukush2019}, which is generated by neighbourhoods of the form $E_S = \lbrace x \in \calX \, : \, \langle Sx,x\rangle_{\calX} < 1 \rbrace$, for some $S \in L_1^{+}(\calX)$.  As this topology is coarser than the standard norm-induced topology on $\calX$, this imposes stronger smoothness constraints on $k$ than standard continuity. For example, the kernel $k(x,y) = \exp(-\frac{1}{2}\norm{x-y}_{\calX}^{2})$ is not the Fourier transform of any Borel measure on $\calX$ \citep[Example 4.1]{Kukush2019} and so does not satisfy the requirements of Theorem \ref{thm:Tik_KSD}.

\begin{example}\label{exp:ker_spec}
    If $T\in L^{+}_{1}(\calX)$ then the SE-$T^{1/2}$ kernel is the characteristic function of $N_{T}$ and the IMQ-$T^{1/2}$ kernel is the characteristic function of the measure which corresponds to the random element $\eta X$ where $X\sim N_{T}$ and $\eta\sim N(0,1)$ where $N(0,1)$ is the standard Gaussian on $\bbR$ and $\eta,X$ are independent \citep[Section 5]{Wynne2022}.
\end{example}

We now combine the Fourier representations obtained in Theorem \ref{thm:Tik_KSD} with the measure equation characterisation results Proposition \ref{prop:fourier_solution} and Proposition \ref{prop:unique} to obtain conditions for when KSD can separate measures. As mentioned before, this result holds in greater generality than previous KSD separation results since we are dealing in the potentially infinite dimensional case where the existing proofs based upon probability density functions cannot be applied. 

\begin{theorem}\label{thm:spec_separates}
    Suppose Assumptions \ref{ass:X} and  \ref{ass:P_Q_assumptions} and $k(x,y) = \widehat{\mu}(x-y)$ for some $\mu\in\calB(\calX)$ with full support. If $k$ satisfies Assumption \ref{ass:nonvec_k_ass} then
    \begin{align*}
    \emph{KSD}_{\calA,k}(Q,P) = 0\iff Q=P,
    \end{align*}
    and if $k$ satisfies Assumption \ref{ass:vec_k_ass} then 
    \begin{align*}
    \emph{KSD}_{\calA_v,K}(Q,P) = 0\iff  Q=P.
    \end{align*}
\end{theorem} 

\begin{example}
    If $T\in L^{+}_{1}(\calX)$ is injective then $N_{T}$ has full support \citep[Proposition 1.25]{DaPrato2006} and so by Example \ref{exp:ker_spec} the SE-$T^{1/2}$ and IMQ-$T^{1/2}$ kernels correspond to measures with full support so both induce a KSD that separates $Q$ from $P$.
\end{example}

As mentioned before, the condition that $k$ is the Fourier transform of a measure is quite restricting. For example, the SE-$T$ only satisfies this if $T$ is trace class which precludes for example $T=I_{\calX}$ the identity operator. This is because the random variable which would correspond to SE-$I_{\calX}$ would have variance $1$ when projected to any basis direction and hence would have infinite norm almost surely and not lie in $\calX$. To extend Theorem \ref{thm:spec_separates} to a wider range of kernels we use a limiting argument, analogous to that of \citet[Theorem 4]{Wynne2022}. More specifically, the idea is to write the kernels as a limit of kernels using hyperparameters $T_{n}$ which are all trace class and to then show that the KSD using $T_{n}$ is a limit of the KSD using $T$.

\begin{theorem}\label{corr:separates}
Suppose Assumptions \ref{ass:X} and \ref{ass:P_Q_assumptions} 
and that $T\in L(\calX)$ is such that $T^{*}$ is surjective. If $k$ is either the SE-$T$ or IMQ-$T$ kernel then 
    \begin{align*}
        \emph{KSD}_{\calA_v,K}(Q,P) = 0 & \iff Q=P.
    \end{align*}
\end{theorem}

This yields flexibility in the choice of $T$ which acts as a hyper-parameter. Some choices of $T$ in the context of kernel-based two-sample testing for functional data were investigated in \citet{Wynne2022}. The theorem only includes the vectorised case as the proof technique does not directly apply to the non-vectorised case, see the proof in the supplement \citep{Wynne2023Supp} for more discussion.

\section{Numerical simulations}\label{sec:numerics}

This section describes the methodology of using KSD to perform statistical tests. Synthetic examples are given for goodness-of-fit testing for Gaussian and non-Gaussian targets and an illustrative example of the ability of KSD to quantify the simulation error path sampling algorithms. Code for all the experiments may be found at \url{https://github.com/georgewynne/Infinite_Dimensional_KSD}.

\subsection{Testing methodology}
The KSD goodness-of-fit testing framework established in \citet{Chwialkowski2016,Liu2016} will be adopted which is now outlined. For the rest of this section we shall only investigate the vectorised operator $\calAv$ version of KSD due to it being easier to implement since it involves less derivatives. We stress that the same testing methodology applies to the non-vectorised case too. 

Given i.i.d.\ samples $\{X_{i}\}_{i=1}^{n}$ from $Q$ consider the $U$-statistic
\begin{align}
	\widehat{\text{KSD}}_{\calAv,K}(Q,P)^{2} = \frac{1}{n(n-1)}\sum_{1\leq i\neq j\leq n}h_{v}(X_{i},X_{j}),\label{eq:U_stat}
\end{align}
which is an unbiased estimator of \eqref{eq:vec_double_int}, where $h_{v}$ is the Stein kernel formed using the vectorised Stein operator. From standard $U$-statistic theory \citep{Serfling1980}, assuming $\bbE_{Q}[h_{v}(X,X')^{2}]<\infty$ the limiting distributions under the null and alternative can be derived, but are hard to simulate from. To alleviate this it is standard to use a bootstrap procedure \citep{Liu2016}
\begin{align*}
	\widehat{\text{KSD}}_{\calAv,K}(Q,P)^{2}_{B} = \frac{1}{n^{2}}\sum_{1\leq i\neq j\leq n}(w_{i}-1)(w_{j}-1)h(X_{i},X_{j}),
\end{align*}
where $w_{1},\ldots,w_{n}\sim \text{Multi}(n;1/n,\ldots,1/n)$. Then, after generating multiple bootstrap samples, the user will reject the null hypothesis if the test statistic \eqref{eq:U_stat} falls outside a certain percentile of the empirical histogram of the bootstrap samples, otherwise the null hypothesis is not rejected. Under the assumptions outlined above which ensure KSD is positive when $Q\neq P$, the limiting distributions imply that the test is consistent in the sense that when $Q\neq P$ the power converges to one in the limit of more data \citep[Proposition 4.2]{Liu2016}. This logic applies in the present potentially infinite dimensional case since $U$-statistic theory only examines the randomness of $h_{v}$ evaluated on the samples, which is a real-valued random variable, and is agnostic of the space the samples themselves lie in \cite{Huskova1993}. Additionally, the result holds in the limit of number of data points and number of bootstrap repetitions therefore the nominal size of the tests, set below at $5\%$ will not necessarily be observed exactly.

The computational cost of this test is $O(n^{2}BH)$ where $n$ is the number of data points, $B$ is the number of bootstrap repetitions and $H$ is the cost of evaluating $h_{v}$ on any pair of points. This final cost $H$ will depend on the kernel being used and which numerical methods are used to calculate functional norm terms. For example, if $\calX = L^{2}([0,1])$ then inner products for this space will need to be calculated and different quadrature methods have different costs. A version of the KSD test which has linear cost in $n$ has been investigated by \citet{Wittawat2017}.

While there exists multiple two-sample functional testing frameworks, see for example \citet{Wynne2022} and references therein, as far as the authors are aware, there does not currently exist one-sample goodness-of-fit tests for non-Gaussian Gibbs measures as studied in this paper. This is a great strength of KSD since non-Gaussian Gibbs measures on functional spaces are often hard to sample from. Therefore, if one only had a two-sample goodness-of-fit test then one would have to perform this difficult sampling problem to be able to perform the test, which could make the overall testing process infeasible. To make comparison to existing methods which cannot be applied in the one-sample Gibbs case we first investigate the Gaussian case. After this a non-Gaussian Gibbs measure case is investigated, for which the methods used in the Gaussian case are not applicable. In this case only the performance of KSD is investigated.

\subsection{Synthetic data goodness-of-fit experiments}

The SE-$\gamma^{-1}T$ kernel $k_{\text{SE}}(x,y) = e^{-\frac{1}{2\gamma^{2}}\norm{Tx-Ty}_{\calX}^{2}}$ and IMQ-$\gamma^{-1}T$ kernel $k_{\text{IMQ}}(x,y) = (\gamma^{-2}\norm{Tx-Ty}_{\calX}^{2}+1)^{-1/2}$ are used with $\gamma\in\bbR$ chosen via the median-heuristic. This means $\gamma = \text{Med}\{\norm{TX_{i}-TX_{j}}_{\calX}, 1\leq i\neq j\leq n\}$ where $\{X_{i}\}_{i=1}^{n}$ are the i.i.d.\ samples from the unknown measure $Q$. This is a commonly used heuristic in kernel-based statistical methods. 

Two choices of $T$ will be used, $T_{1} = I_{\calX}$, the identity operator and $T_{2}x = \sum_{i=1}^{\infty}\eta_{i}\langle x,e_{i}\rangle_{\calX}e_{i}$ where $\eta_{i} = \lambda_{i}^{-1}$ for $1\leq i\leq 50$ and $\eta_{i} = 1$ for $i>50$ with $e_{i},\lambda_{i}$ the eigensystem of Brownian motion. This choice of $T_{2}$ will emphasise, in an increasing manner, higher frequency activity with respect to the Brownian motion basis. The cut off at the 50th frequency, as opposed to applying this whitening operator to all frequencies, is done so that most of the signal will be impacted while ensuring that $\norm{T_{2}(X_{i}-X_{j})}_{\calX}$ is finite with probability one, which would not be the case if all frequencies were whitened and would thus make the tests invalid. 

The purpose of these two choices is to investigate the effect of treating the original data without any changes, using $T_{1}$, and to use a norm in the kernel that emphasises high frequency deviations from the target measure, $T_{2}$. In general identifying an optimal choice of hyperparameter is an open problem and we believe it an important research question. 

The performance of the tests is recorded in tables. When the test is for the null hypothesis, the closest value to the nominal $5\%$ level is bold, for the other experiments the highest power value is bold.

\subsection*{Brownian motion target}
Let $\calX = L^{2}([0,1])$ and the target measure $P=N_{C}$, meaning $U = 0$, shall be Brownian motion over $[0,1]$ so the covariance operator $C$ has eigenvalues $\lambda_{i} = (i-0.5)^{-2}\pi^{-2}$ and eigenfunctions $e_{i}(t) = \sqrt{2}\sin((i-0.5)\pi t)$. The number of samples $n$ will be specified in each experiment, $2000$ bootstraps are performed to calculate the rejection threshold in each experiment and each test is repeated $500$ times to calculate test power and the nominal type one error rate is set to 5$\%$.

In the specification of the experiments $B_t$ denotes standard Brownian motion. All function samples are observed at $100$ points on a uniform grid across $[0,1]$ so function reconstruction is not required. We compare to a small-ball probability based method of \citep{Bongiorno2018}, a Cram\'{e}r von-Mises test using spherical projections \citep{Ditzhaus2018} and a Cram\'{e}r von-Mises test based on Gaussian process projections \citep{Bugni2009}. Experiment \ref{exp:null} represents the null hypothesis and therefore will quantify type one error, Experiments \ref{exp:clip}, \ref{exp:OU}, \ref{exp:AC1}, \ref{exp:AC2} were studied in \citet{Bongiorno2018} and Experiments \ref{exp:2B}, \ref{exp:B+t} were studied in \citet{Ditzhaus2018}. 

\begin{enumerate}
    \item \label{exp:null} $n=50$ and $Q$ is the law of Brownian motion.
    
    \item \label{exp:clip} $n=50$ and $Q$ is the law of the Brownian motion clipped to $5$ frequencies $\sum_{i=1}^{5}\lambda_{i}^{1/2}\xi_{i}e_{i}$ with $\xi_{i}\overset{i.i.d.}{\sim} N(0,1)$ and $\lambda_{i},e_{i}$ from the eigensystem of $C$ as discussed above.  
    
    \item \label{exp:OU}$n = 25$ and $Q$ is the law of the Ornstein-Uhlenbeck process $$\diff X_t = 0.5(5-X_t)\diff t + \diff B_t.$$
    
    \item \label{exp:AC1} $n = 50$ and $Q$ is the law of $X_t = (1+t^{2})B_t$.

     \item \label{exp:AC2} $n = 50$ and $Q$ is the law of $X_t = (1+\sin(2\pi t))B_t$.
    
    \item \label{exp:2B} $n = 25$ and $Q$ is the law of $2B_t$.
    
    \item \label{exp:B+t} $n = 25$ and $Q$ is the law of $B_t + 1.5t(t-1)$. 
\end{enumerate}

Tables \ref{tab:SB} and \ref{tab:CvM} show that the kernel Stein methodology has superior, or at least comparable, performance while maintaining a controlled type one error across all kernel choices. Experiment \ref{exp:clip} clearly demonstrates the advantage of the $T_{2}$ hyperparameter as it emphasises higher frequencies enough to be able to easily detect the difference between the clipped and standard Brownian motion signals. Experiment \ref{exp:OU} shows all the kernel choices achieve strong performance compared to the small-ball probability based method. Experiments \ref{exp:AC1} and \ref{exp:AC2} also show the kernel based methods perform well. However Experiment \ref{exp:AC2} shows poor performance for $T_{2}$ due to the interaction between the eigenbasis of Brownian motion consisting of sine functions and the deviation in Experiment \ref{exp:AC2} consisting of a sine function. Experiment \ref{exp:2B} shows that the SE kernel has good performance in covariance scale detection whereas the IMQ does not, this is a known phenomenon from finite dimensional investigations \citep{Gretton2012}. Experiment \ref{exp:B+t} shows strong performance in a mean shift detection for the kernel methods. 

\begin{table}[ht]
\begin{center}
 \begin{tabular}{c| c c c c c}
 \toprule
 Experiment & SE-$T_{1}$ &  SE-$T_{2}$ & IMQ-$T_{1}$ & IMQ-$T_{2}$ & SB \\ 
 \midrule
 1  & 0.06 & \textbf{0.05} & 0.052 & 0.048 & 0.032 \\
 2  & 0.056 & \textbf{1.0}  & 0.054 & 0.952 & 0.615 \\
 3  & \textbf{1.0} & \textbf{1.0} & \textbf{1.0} & \textbf{1.0} & 0.023 \\
 4  & 0.9 & \textbf{1.0} & 0.554 & 0.986 & \textbf{1.0} \\
 5  & \textbf{1.0} & 0.542 & \textbf{1.0} & 0.134 & 0.05 \\
 \bottomrule
\end{tabular}
\caption{Proportion of times the null was rejected on Experiments \ref{exp:null}-\ref{exp:AC2}, SB denotes the small-ball probability method of \citet{Bongiorno2018}.}
\label{tab:SB}
\end{center}
\end{table}

\begin{table}[ht]
\begin{center}
 \begin{tabular}{c| c c c c c c} 
 \toprule
 Experiment & SE-$T_{1}$ &  SE-$T_{2}$ & IMQ-$T_{1}$ & IMQ-$T_{2}$ & CvM SP & CvM GP\\ \midrule
 6  & 0.858 & 0.786 & 0.332 & 0.206 &  \textbf{0.895} & 0.763\\
 7  & 0.522 & \textbf{0.99} & 0.608 & 0.87 &  0.98 & 0.858\\
 \bottomrule
\end{tabular}
\caption{Proportion of times the null was rejected on Experiments \ref{exp:2B}-\ref{exp:B+t}, CvM SP denotes the Cram\'{e}r von-Mises test based on spherical projections of \citet{Ditzhaus2018} and CvM GP denotes the Cram\'{e}r von-Mises test based on Gaussian process projections of \citet{Bugni2009}.}
\label{tab:CvM}
\end{center}
\end{table}

\subsection*{Gibbs measure target}
Let $\calX = L^{2}([0,50])$ and set the target measure $P$ to be a conditioned version of the non-linear SDE over the interval $[0,50]$
\begin{align}
    \diff X_{t} = 0.7\sin(X_{t})\diff t + \diff B_{t},\label{eq:non_lin_SDE}
\end{align}
where $B_{t}$ is the driving Brownian motion. The paths shall be conditioned so that $X_{0} = X_{50} = 0$. This conditioned diffusion was studied in \citet{Bierkens2021} and is an example of a bridge diffusion. Bridge diffusions are common in applied mathematics, for example in finance, econometrics and molecular dynamics \citep{Roberts2001,Beskos2006,Pinski2010}. 

The base measure $N_{C}$ is the Brownian bridge over $[0,50]$ and by Girsanov's theorem, 
\begin{align*}
    U(x) & = \frac{1}{2}\int_{0}^{50}0.49\sin(x(s))^{2} + 0.7\cos(x(s))ds\\
    DU(x) & = 0.49\sin(x(\cdot))\cos(x(\cdot)) -0.35\sin(x(\cdot)).
\end{align*}
A reader can consult \citet{Bierkens2021} for the derivation to get the expression for $U$. We use the piecewise-deterministic Markov process sampler from \citet{Bierkens2021} to simulate samples. Deviations from the target distribution will be a deterministic drift given by $Y_{t} = X_{t} +\delta t/50$, for $\delta\in\bbR$ so $\delta = 0$ represents the null hypothesis. Each test uses $n = 100$ samples and is repeated $100$ times to compute the size and power, each trajectory is observed on a uniform grid of $129$ points over $[0,50]$. 

Table \eqref{tab:SDE} shows the performance of the tests. The size of the test is slightly inflated when using the $T_{1}$ hyperparameter. The deviations from the null when $\delta > 0$ can be identified by all configurations of the kernel and there is little difference between the power performance of the different configurations. 

\begin{table}[ht]
\begin{center}
 \begin{tabular}{c |c c c c} 
 \toprule
 $\delta$ & SE-$T_{1}$ &  SE-$T_{2}$ & IMQ-$T_{1}$ & IMQ-$T_{2}$ \\ 
 \midrule
 0  & 0.08 & \textbf{0.05} & 0.07 & \textbf{0.05}\\
  0.05  & 0.17 & 0.17 & \textbf{0.20} & 0.18 \\
  0.1  & 0.43 & 0.4 & \textbf{0.45} & 0.43\\
  0.15  & \textbf{0.81} & 0.77 & 0.79 & 0.77\\
  0.2  & \textbf{0.97} & 0.95 & 0.96 & 0.96\\
 \bottomrule
\end{tabular}
\caption{Proportion of times the null was rejected on the non-linear conditioned SDE experiment, $\delta$ denotes the parameter controlling the deviation from the null.}
\label{tab:SDE}
\end{center}
\end{table}

\subsection*{Euler-Maruyama discretisation error}
With the same scenario as the previous testing experiment for the conditioned non-linear SDE \eqref{eq:non_lin_SDE} the ability of KSD to give a measurement of simulation accuracy is measured when the Euler-Maruyama method along with a rudimentary accept/reject step is used to simulate the conditioned diffusion. The purpose of this experiment is to show the utility of KSD outside of a testing framework by using it solely as a measure of discrepancy between a target and a candidate distribution, without requiring samples from that target distribution.

We simulate approximations to the conditioned non-linear SDE \eqref{eq:non_lin_SDE} by using the \sloppy{Euler-Maruyama} (EM) method with varying number of steps and accepting the trajectories for which $\abs{X(50)} < \varepsilon=0.1$. In practice one would not use this rudimentary simulation method to sample from the conditioned SDE however for our demonstration purposes this is a simple way of seeing how sensitive KSD is to discretisation error when simulating random functions. The range of EM steps that are used is $\{5,10,15,20,25\}$ then the trajectories are linearly interpolated over a regular grid of $100$ points over $[0,50]$ to give the simulated trajectories. We take $2000$ simulations for each number of EM steps to make our estimates of KSD.  

The simulation method is not exact due to $\varepsilon > 0$ and we use a biased $V$-statistic estimate of KSD to ensure that the estimated values are positive. Therefore one does not expect the KSD to fall to zero as the number of EM steps increases. However one would expect the KSD value to decrease to some non-zero quantity as the number of EM steps increases. For this experiment the IMQ-$\gamma^{-1}T$ kernel for $T \in\{T_{1},T_{2}\}$ is used with $\gamma = 1$ for simplicity.  

In Figure \ref{fig:KSD_EM} we see that for both choices of $T$ the KSD value decreases to a positive constant. As expected $T_{2}$ is more discerning of the choice of steps in the sense that there is greater difference in the KSD values for each different number of steps. This is to be expected since $T_{2}$ is designed to be more discerning of higher frequency fluctuations of the signals whereas $T_{1}$ is not. Hence why for $T_{1}$ all step sizes greater than $5$ have a similar KSD but the KSD continues to decrease for $T_{2}$. This figure shows  that increasing the resolution of the EM method past $15$ steps is not overall increasing the similarity of the simulated data compared to the target distribution and for finer resolutions the inexact accept/reject step is causing the difference between the simulated samples and target distribution. This example shows that for this specific experiment the $T_{2}$ parameter provides a more discerning KSD than $T_{1}$. Though this will not always be the case, since experiments will always depend on the target measure of interest, the construction of $T_{2}$ was made to be sensitive to high-frequency deviations from the target measure and hence can be seen as a reasonable choice when identifying the quality of a sampling method.

\begin{figure}
    \centering
    \includegraphics[scale = 0.25]{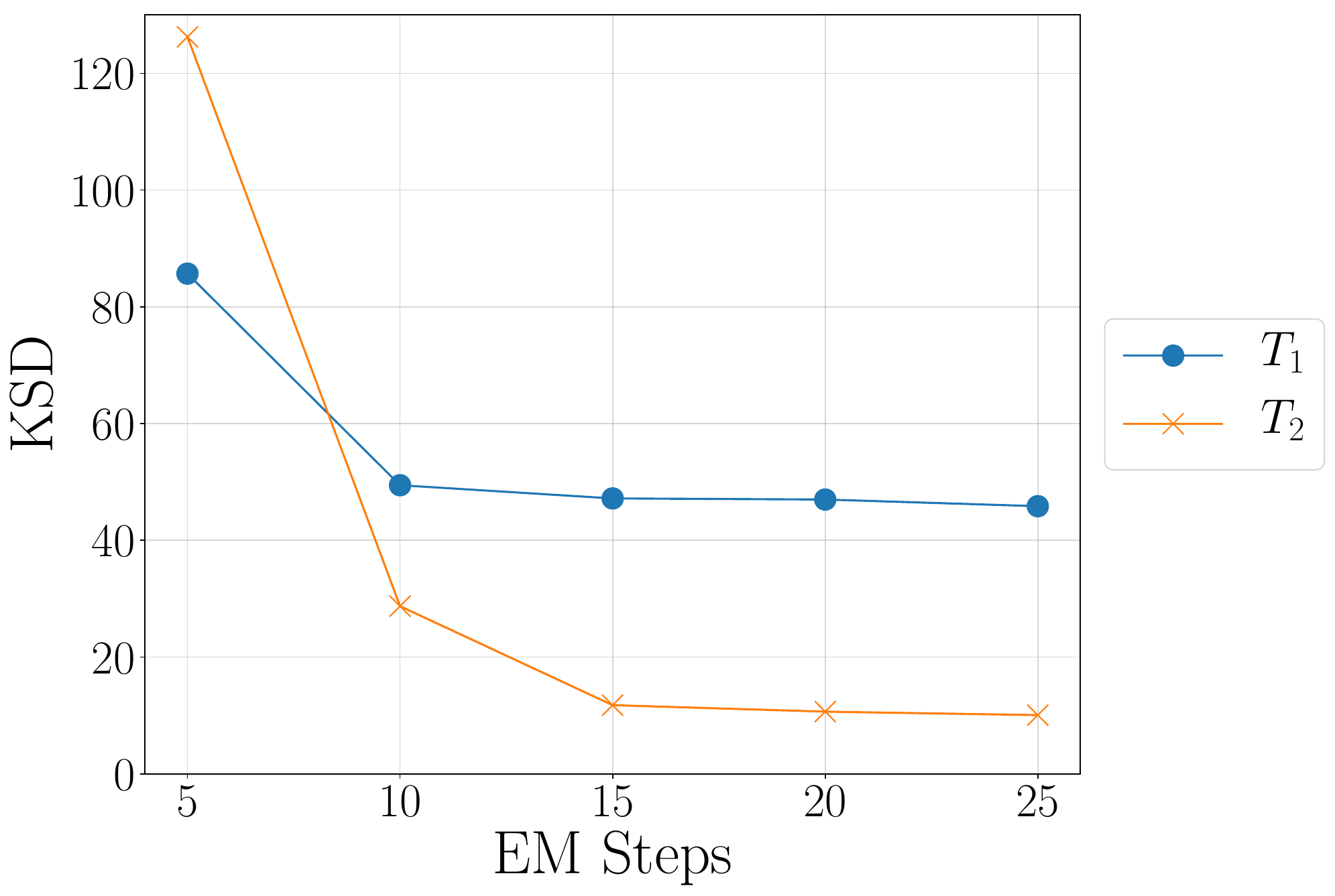}
    \caption{A plot of KSD using the IMQ kernel against the number of steps in the Euler-Maruyama simulation to simulate the target measure. The target measure is the conditioned SDE \eqref{eq:non_lin_SDE}. The KSD value was estimated using $2000$ samples of the Euler-Maruyama simulation, keeping the trajectories with $\abs{X(50)}<0.1$.}
    \label{fig:KSD_EM}
\end{figure}

\section{Conclusion}\label{sec:conclusion}
We have formulated kernel Stein discrepancy for measures on infinite dimensional, separable Hilbert spaces. A Fourier representation was derived in Theorem \ref{thm:Tik_KSD} from which conditions were identified to ensure that KSD can separate measures. This addresses both main aims of the paper, to theoretically justify the extension of the KSD methodology to infinite dimensional data and to obtain a more interpretable representation of KSD in which the action of the Stein operator and kernel are distinct. The derivations also hold in the finite dimensional case. Numerical simulations were performed in Section \ref{sec:numerics} which validate the utility of the KSD approach for infinite dimensional data. There are many further questions which we believe are outside of the scope of this paper. First, the generalisation of KSD beyond the base Gaussian measure case. Our approach using the generator method with infinite dimensional Langevin diffusions revolves around specific, sophisticated results which require a base Gaussian measure \citep{Albeverio1999uniqueness,Bogachev1995regularity}. Therefore we believe a distinct approach must be taken to this problem. Second, a central question is the topology that KSD imbues upon the space of measures over $\calX$. There have been many investigations on the topological properties of KSD in the finite dimensional context. A recent addition by \citet{Barp2022} makes clear many results in the finite dimensional case. These depend on densities and associated score functions of measures and so it is unclear how far the results can be adapted to the infinite dimensional case.

Future work on this topic would involve applying the Fourier representation as a tool to answer questions regarding KSD and analyse algorithms that use KSD. For example, a central question regarding KSD is when used as a distance between probability measures does it metrise the weak topology \citep{Gorham2017}. This would likely translate into a question on the tails of the Fourier measure of a kernel both in the finite and infinite dimensional case. Another example is the variational inference technique using KSD, called Stein variational gradient descent (SVGD) \citep{Liu2016SVGD}, which could have its dynamics analysed using the Fourier representation.

\section*{Acknowledgements}

The authors would like to thanks Sebastiano Grazzi for help adapting code from \citet{Bierkens2021} and Chris J. Oates for recommending the plot in Figure \ref{fig:Test_Functions}. Thanks also the the associate editor and reviewers for helpful comments to improve this paper. The research of George Wynne was conducted when at Imperial College London and was supported by an EPSRC Industrial CASE award [EP/S513635/1] in partnership with Shell UK Ltd. The research of Miko{\l}aj Kasprzak was supported by the FNR grant FoRGES (R-AGR- 3376-10) at Luxembourg University and the European Union’s Horizon 2020 research and innovation programme under the Marie Sk{\l}odowska-Curie grant agreement [101024264- Stein-ML]. Andrew B. Duncan was supported by Wave 1 of The UKRI Strategic Priorities Fund under the EPSRC Grant EP/T001569/1 and EPSRC Grant EP/W006022/1, particularly the “Ecosystems of Digital Twins” theme within those grants \& The Alan Turing Institute.

{
\bibliographystyle{abbrvnat}
\bibliography{inf_dim_refs.bib}
}

\appendix
\section{Appendix}
\setcounter{section}{1}
\subsection{Derivatives of measures and integration-by-parts} \label{subsec:deriv_measures}
At its core, Stein's lemma is a direct application of the integration-by-parts formula.   In the infinite dimensional setting, there are some additional hurdles which must be overcome to obtain a suitable generalisation.  The key challenge arises from the lack of an infinite-dimensional analogue to the Lebesgue measure and its translation invariant properties. Therefore, one cannot appeal to the divergence theorem for Lebesgue measures to recover an integration-by-parts formula. To alleviate this we must use logarithmic derivatives of measures \citep[Chapter 6, Chapter 7]{Bogachev2010}. Our presentation of this concept is focused on our particular context, where $\calX$ is a separable Hilbert space. Logarithmic gradients can be used in far more generality but we leave the interested reader to the references to explore this, in particular \citep[Chapter 6]{Bogachev2010}. We start with the notion of Fomin differentiable measures, which play the role of derivatives of density functions in finite dimensions.

\begin{definition}\label{def:meas_diff}
    A measure $\mu$ on $\calX$ is called differentiable along a vector $h\in\calX$ if, for every set  $A$ in the Borel sigma algebra, there exists a finite limit
    \begin{align*}
        d_{h}\mu(A) \coloneqq \lim_{t\rightarrow 0}\frac{\mu(A+th)-\mu(A)}{t}.
    \end{align*}
\end{definition}
The resulting measure $d_{h}\mu$ is absolutely continuous with respect to $\mu$ \citep[Corollary 3.3.2]{Bogachev2010}. Its Radon-Nikodym derivative is denoted $\beta^{\mu}_{h}$ and is known as the logarithmic derivative of $\mu$ along $h$. An example is when $\mu = N_{C}$ is a centred Gaussian measure, where the vectors along which $\mu$ is differentiable is the Cameron-Martin space \citep[Theorem 3.1.9]{Bogachev2010} and $\beta_{Ch}^{N_{C}}(x) = -\langle h,x\rangle_{\calX}$. The logarithmic derivative is the infinite dimensional analogue to a score function, used often in statistics. For measures absolutely continuous with respect to Gaussians one can use \citet[Corollary 6.1.4]{Bogachev2010} to deduce the logarithmic derivative. For example, $P = e^{-U}N_{C}$ has logarithmic derivative $\beta_{Ch}^{P}(x) = -\langle h,x+CDU(x)\rangle_{\calX}$.

These logarithmic derivatives can then be used to form integration by parts results. The following result is \citet[Theorem 6.1.2]{Bogachev2010} adapted to our context and $\partial_{h}f$, for $f\colon\calX\rightarrow\bbR$, shall denote the directional derivative in direction $h$ which is equal to $\langle h,Df(x)\rangle_{\calX}$ when $f$ is Fr\'{e}chet differentiable. 
\begin{proposition}\label{prop:IBP}
    Let $h\in\calX$ and suppose $\mu$ is differentiable in the sense of Definition \ref{def:meas_diff} in direction $h$ and that $f$ is a $\mu$-integrable function such that for $\mu$-almost all $x\in\calX$ the function $t\rightarrow f(x+th)$ is differentiable, $f\in L^{1}(\calX;d_{h}\mu)$ and $\partial_{h}f\in L^{1}(\calX;\mu)$ then
    \begin{align*}
        \int_{\calX}\partial_{h}f(x)d\mu(x) = -\int_{\calX}f(x)\beta_{h}^{\mu}d\mu(x).
    \end{align*}
\end{proposition}

\subsection{Proofs for Section \ref{sec:KSD_for_GPs}}
In this section the well-posedness of the operator \eqref{eq:A_operator} is established as well as the conclusion of Theorem \ref{thm:KSD_OU}.  Multiple intermediate lemmas will be required to establish that the kernel can reproduce Fr\'{e}chet derivatives. For reproducing derivatives in finite dimensions the corresponding result is well known, see \citet[Lemma 4.34]{Steinwart2008} and \citet{Zhou2008}. The first step is establishing that, given the assumptions in Theorem \ref{thm:KSD_OU}, the partial Fr\'{e}chet derivative of the kernel is in the RKHS. 

\begin{lemma}\label{lem:derivative_in_RKHS}
	Under Assumption \ref{ass:X} if $k$ satisfies Assumption \ref{ass:vec_k_ass}, then $D_{1}k(x,\cdot)[e]\in\calH_{k}$ for all $x,e\in\calX$ and if $k$ satisfies Assumption \ref{ass:nonvec_k_ass} then $D_{1}^{2}k(x,\cdot)[e,u]\in\calH_{k}$ for all $x,e,u\in\calX$.
\end{lemma}
\begin{proof}
The proof is outlined for the first case only as the latter case is done in the exact same way by replacing $k(x,y)$ with $D_{1}k(x,y)$. The approach of \citet[Theorem 1]{Zhou2008}  is generalised to infinite dimensions. By assumption $D_{1}k$ exists and so equals the Gateaux derivative 
\begin{align}
    D_{1}k(x,y)[e] = \lim_{t\rightarrow 0}t^{-1}(k(x+te,y)-k(x,y)).\label{eq:Gateaux}
\end{align}
Now consider $\norm{t^{-1}(k(x+te,\cdot)-k(x,\cdot))}_{k}$. Squaring and using reproducing property
\begin{align*}
    \norm{t^{-1}(k(x+te,\cdot)-k(x,\cdot))}_{k}^{2} & = t^{-2}\big(k(x+te,x+te)-k(x+te,x)\\
    & - (k(x,x+te)-k(x,x))\big)\\
    & = t^{-2}(R(x+te)-R(x)\big),
\end{align*}
where $R(z) = k(z,x+te)-k(z,x)$. As it was assumed that $k$ has continuous first and second order partial Fr\'{e}chet derivatives, the mean value theorem can be applied to $R$
\begin{align*}
    R(x+te)-R(x) & \leq t\norm{e}\sup_{0\leq s\leq 1}\norm{DR(x+ste)}_{L(\calX,\bbR)}\\
    & = t\norm{e}_{\calX}\sup_{0\leq s\leq 1}\norm{D\big(k(\cdot,x+te)-k(\cdot,x)\big)[x+ste]}_{L(\calX,\bbR)}\\
    & = t\norm{e}_{\calX}\sup_{0\leq s\leq 1}\norm{D_{1}k(x+ste,x+te)-D_{1}k(x+ste,x)}_{L(\calX,\bbR)}\\
    & \leq t^{2}\norm{e}_{\calX}^{2}\sup_{0\leq s,h\leq 1}\norm{D\big(D_{1}k(x+ste,\cdot)\big)[x+hte]}_{L(\calX,L(\calX,\bbR))}\\
    & = t^{2}\norm{e}_{\calX}^{2}\sup_{0\leq s,h\leq 1}\norm{D_{2}D_{1}k(x+ste,x+hte)}_{L(\calX,L(\calX,\bbR))} = t^{2}\norm{e}_{\calX}^{2}c_{k},
\end{align*}
where $c_{k}$ is a finite constant by assumption. Therefore $\norm{t^{-1}(k(x+te,\cdot)-k(x,\cdot))}_{k}^{2}\leq \norm{e}_{\calX}^{2}c_{k}$ so the set $\{t^{-1}(k(x+te,\cdot)-k(x,\cdot))\colon t\in\bbR\}$ is contained in a closed ball in $\calH_{k}$. Therefore weak compactness implies that for any sequence $\{t_{n}\}_{n=1}^{\infty}$ converging to zero there exists a subsequence, for which we abuse notation and still denote by $t_{n}$, and some $g\in\calH_{k}$ such that
\begin{align*}
    \lim_{n\rightarrow\infty}t_{n}^{-1}\langle k(x+t_{n}e,\cdot)-k(x,\cdot),f\rangle_{k} = \langle g,f\rangle_{k},
\end{align*}
for all $f\in\calH_{k}$. In particular, taking $f = k(y,\cdot)$ and using the reproducing property implies$g(y) = D_{1}k(x,y)e$ for all $y\in\calX$ by \eqref{eq:Gateaux} which completes the proof. 
\end{proof}

\begin{lemma}\label{lem:conv_in_rkhs}
Under Assumption \ref{ass:X} if $k$ satisfies Assumption \ref{ass:vec_k_ass} then $t^{-1}\big(k(x+te,\cdot)-k(x,\cdot)\big)$ converges in $\calH_{k}$ to $D_{1}k(x,\cdot)[e]$ as $t\rightarrow 0 $ for all $x,e\in\calX$. If $k$ satisfies Assumption \ref{ass:nonvec_k_ass} then $t^{-1}\big(D_{1}k(x+te,\cdot)[u]-D_{1}k(x,\cdot)[u]\big)$ converges in $\calH_{k}$ to $D_{1}^{2}k(x,\cdot)[e,u]$ as $t\rightarrow 0 $ for all $x,e,u\in\calX$. 
\end{lemma}
\begin{proof}
As with the previous proof only the first case is outlined. Again \citet[Theorem 1]{Zhou2008} is followed. Setting $f = D_{1}k(x,\cdot)e$ in the proof of Lemma \ref{lem:derivative_in_RKHS}
\begin{align*}
    \langle D_{1}k(x,\cdot)[e],D_{1}k(x,\cdot)[e]\rangle_{k} & = \lim_{n\rightarrow\infty}t_{n}^{-1}\langle k(x+t_{n}e,\cdot)-k(x,\cdot),D_{1}k(x,\cdot)[e]\rangle_{k}\\
    & =\lim_{n\rightarrow\infty}t_{n}^{-1}\big(D_{1}k(x,x+t_{n}e)[e]-D_{1}k(x,x)[e]\big)\\
    & = D_{2}D_{1}k(x,x)[e,e].
\end{align*}
Therefore,
\begin{align}
    & \norm{t^{-1}\big(k(x+te,\cdot)-k(x,\cdot)\big) - D_{1}k(x,\cdot)[e]}_{k}^{2} \nonumber \\
    & = t^{-2}\big(k(x+te,x+te)-2k(x,x+te)+k(x,x)\big)\label{eq:double_difference}\\
    & - 2t^{-1}\big(D_{1}k(x,x+te)[e] - D_{1}k(x,x)[e]\big) + D_{2}D_{1}k(x,x)[e,e].\label{eq:single_difference}
\end{align}
Next, use the mean value theorem and the dominated convergence theorem to show this converges to zero. The mean value theorem applied to \eqref{eq:double_difference}, applicable due to the continuity of partial Fr\'{e}chet derivatives assumption, gives
\begin{align*}
	k(x+te,x+te) & -2k(x,x+te)+k(x,x) \\
	& = \left[\int_{0}^{1}D_{1}k(x+ste,x+te)-D_{1}k(x+ste,x)ds\right][te],
\end{align*}
and another application of the mean value theorem produces the mixed partial Fr\'{e}chet derivatives
\begin{align*}
k(x+te,x+te) & -2k(x,x+te)+k(x,x)\\
& = \left[\int_{0}^{1}\left[\int_{0}^{1}D_{2}D_{1}k(x+ste,x+ute)du\right][te]ds\right][te].
\end{align*}
By the assumption on the supremum norm of $D_{2}D_{1}k$ we can swap the integrals \citep[Theorem 6]{Diestel1977} and the inputs of the operators
\begin{align*}
	k(x+te,x+te) & -2k(x,x+te)+k(x,x)\\
& = \int_{0}^{1}\int_{0}^{1}D_{2}D_{1}k(x+ste,x+ute)[te,te]duds.
\end{align*}
Proceeding similarly for \eqref{eq:single_difference} reveals
\begin{align*}
    & \norm{t^{-1}\big(k(x+te,\cdot)-k(x,\cdot)\big) - D_{1}k(x,\cdot)[e]}_{k}^{2} \\
    & = \int_{0}^{1}\int_{0}^{1}D_{2}D_{1}k(x+ste,x+ute)[e,e]\\
    & - 2D_{2}D_{1}k(x,x+ste)[e,e] + D_{2}D_{1}k(x,x)[e,e]dsdu.
\end{align*}
The assumption of continuity of the mixed partial Fr\'{e}chet derivatives as well as their boundedness means the dominated convergence theorem implies that the integrands cancel out as $t\rightarrow 0$ which completes the proof. 
\end{proof}

\begin{lemma}\label{lem:derivative_reproduce}
Under Assumption \ref{ass:X} if $k$ satisfies Assumption \ref{ass:vec_k_ass} then $Df(x)$ exists for every $f\in\calH_{k},x\in\calX$ and $\langle D_{1}k(x,\cdot)[e],f\rangle_{k} = Df(x)[e]$ for all $x,e\in\calX$. If $k$ satisfies Assumption \ref{ass:nonvec_k_ass} then $D^{2}f(x)$ exists for every $f\in\calH_{k},x\in\calX$ and $\langle D_{1}^{2}k(x,\cdot)[e,u],f\rangle_{k} = D^{2}f(x)[e,u]$ for all $x,e,u\in\calX$.
\end{lemma}
\begin{proof}
As before only the first case is outlined explicitly with the latter following analogously. By Lemma \ref{lem:conv_in_rkhs}
\begin{align*}
    \langle D_{1}k(x,\cdot)[e],f\rangle_{k} & = \lim_{t\rightarrow 0}\langle t^{-1}\big(k(x+te,\cdot)-k(x,\cdot)\big),f\rangle_{k} = \lim_{t\rightarrow 0}t^{-1}\big(f(x+te)-f(x)\big),
\end{align*}
so the Gateaux derivative $D_{G}f(x)[e]$ exists and is equal to the desired inner product. To show Fr\'{e}chet differentiability it remains to show the Gateaux derivative is continuous. Start with the bound
\begin{align*}
    \norm{D_{G}f(x)-D_{G}f(y)}_{L(\calX,\bbR)} & = \sup_{\norm{e}_{\calX}=1}\lvert\langle(D_{1}k(x,\cdot)-D_{1}k(y,\cdot))[e],f\rangle_{k}\rvert\\
   &\leq  \norm{f}_{k}\sup_{\norm{e}_{\calX}=1}\norm{(D_{1}k(x,\cdot)-D_{1}k(y,\cdot))[e]}_{k}.
\end{align*}
Expanding the RKHS norm
\begin{align*}
    \norm{(D_{1}k(x,\cdot)-D_{1}k(y,\cdot))[e]}_{k}^{2} & = D_{2}D_{1}k(x,x)[e,e]-D_{2}D_{1}k(x,y)[e,e] \\
    & + D_{2}D_{1}k(y,y)[e,e]-D_{2}D_{1}k(y,x)[e,e]\\
    & \leq \norm{e}_{\calX}^{2}\big(\norm{D_{2}D_{1}k(x,x)-D_{2}D_{1}k(x,y)}_{L(\calX\times\calX,\bbR)}\\
    & +\norm{D_{2}D_{1}k(y,y)-D_{2}D_{1}k(y,x)}_{L(\calX\times\calX,\bbR)}\big).
\end{align*}
Therefore 
\begin{align*}
    \norm{D_{G}f(x)-D_{G}f(y)}_{L(\calX,\bbR)} & \leq \norm{f}_{k}\big(\norm{D_{2}D_{1}k(x,x)-D_{2}D_{1}k(x,y)}_{L(\calX\times\calX,\bbR)}\\
    & +\norm{D_{2}D_{1}k(y,y)-D_{2}D_{1}k(y,x)}_{L(\calX\times\calX,\bbR)}\big)^{1/2},
\end{align*}
and the quantity on the right hand side converges to zero as $\norm{x-y}_{\calX}\rightarrow 0$ by the continuity assumption on $D_{2}D_{1}k$, this completes the proof. 
\end{proof}

\subsubsection{Proof of Lemma \ref{lemm:well_defined}}
First of all, Lemma \ref{lem:derivative_reproduce} implies the first and second order Fr\'{e}chet derivatives that are required in $\calA$ exist. Using the reproducing property, for any orthonormal basis $\{e_{i}\}_{i=1}^{\infty}$ of $\calX$ 
\begin{align*}
D^{2}f(x)[e_{i},e_{i}] = \langle D^{2}_{1}k(x,\cdot)[e_{i},e_{i}],f\rangle_{k} & \leq \norm{D_{1}^{2}k(x,\cdot)[e_{i},e_{i}]}_{k}\norm{f}_{k}\\
& = \sqrt{D_{2}^{2}D_{1}^{2}k(x,x)[e_{i},e_{i},e_{i},e_{i}]}\norm{f}_{k}\\
& \leq C_{k}\norm{f}_{k}
\end{align*}
for a constant $C_{k}$ by Assumption \ref{ass:nonvec_k_ass}. Therefore taking the basis to be the eigensystem of $C$
\begin{align*}
    \bbE_{Q}[\text{Tr}[CD^{2}f(X)]]=\bbE_{ Q}\left[\sum_{i=1}^{\infty}\langle CD^{2}f(X),e_{i}\rangle_{\calX}\right] & = \bbE_{Q}\left[\sum_{i=1}^{\infty}\lambda_{i}D^{2}f(X)[e_{i},e_{i}]\right] \\
    & \leq C_{k}\norm{f}_{k}\sum_{i=1}^{\infty}\lambda_{i} = C_{k}\norm{f}_{k}\text{Tr}[C] < \infty.
\end{align*}
For the $\calAv$ case
\begin{align*}
    DF(x)[e_{i},e_{i}] = DF_{i}(x)[e_{i}] = \langle D_{1}k(x,\cdot)[e_{i}],F_{i}\rangle_{k} & \leq \sqrt{D_{2}D_{1}k(x,x)[e_{i},e_{i}]}\norm{F_{i}}_{k}\\
    & \leq C_{k}'\norm{F_{i}}_{k}\\
    & \leq C_{k}'\norm{F}_{K},
\end{align*}
for some constant $C_{k}'$ by Assumption \ref{ass:vec_k_ass}. Therefore, similar to the $\calA$ case, the trace norm term is bounded by $\bbE_{Q}[\text{Tr}[CDF(X)]]\leq C_{k}'\norm{F}_{K}\text{Tr}[C]$. 

Now the expectation of the inner product term in $\calA$ and $\calAv$ needs to be shown to be finite. In the former, by the reproducing property 
\begin{align*}
    \bbE_{Q}[\langle Df(X),X+CDU(X)\rangle_{\calX}] & = \bbE_{Q}[\langle D_{1}k(X,\cdot)[X+CDU(X)],f\rangle_{k}]\\
    & \leq \bbE_{Q}[D_{2}D_{1}k(X,X)[X+CDU(X),X+CDU(X)]^{1/2}]\norm{f}_{k}\\
    & \leq C_{k}''\norm{f}_{k},
\end{align*}
for a constant $C_{k}''$ by Assumption \ref{ass:nonvec_k_ass}. For the $\calAv$ case
\begin{align*}
    \bbE_{Q}[\langle F(X),X+CDU(X)\rangle_{\calX}] & = \bbE_{Q}[\langle K(X,\cdot)[X+CDU(X)]\rangle_{K}]\\
    & = \bbE_{Q}[k(X,X)^{1/2}\norm{X+CDU(X)}_{\calX}]\norm{F}_{K} < \infty,
\end{align*}
where the last equality is by the reproducing property of operator-valued kernels. The final inequality is by Assumption \ref{ass:vec_k_ass} and Assumption \ref{ass:P_Q_assumptions}. This completes the proof that the two KSD expressions are well-defined. 

\subsubsection{Proof of Theorem \ref{thm:KSD_OU}}
The proof for the $\calA$ case is done first. By the derivative reproducing properties established in Lemma \ref{lem:derivative_reproduce} 
\begin{align*}
    \calA f(x) & = \text{Tr}[CD^{2}f(x)] - \langle x+CDU(x),Df(x)\rangle_{\calX}\\
    & = \sum_{i=1}^{\infty}\lambda_{i}D^{2}f(x)[e_{i},e_{i}] - Df(x)[x+CDU(x)]\\
    & = \langle f,\sum_{i=1}^{\infty}\lambda_{i}D^{2}_{1}k(x,\cdot)[e_{i},e_{i}] - D_{1}k(x,\cdot)[x+CDU(x)]\rangle_{k}\\
    & \eqqcolon \langle f,\xi(x)\rangle_{k}.
\end{align*}
Expectation with respect to $Q$ will need to be swapped with the $\calH_{k}$ inner product to obtain the desired result. To do this we need to show $\bbE_{Q}[\norm{\xi(X)}_{k}]<\infty$ since then $\xi$ is Bochner integrable \citep[Theorem 2.6.5]{Hsing2015} and so the expectation and inner product in the definition of $\text{KSD}_{\calA,k}$ can be swapped \citep[Theorem 3.1.7]{Hsing2015}. Using the reproducing properties
\begin{align*}
    \norm{\xi(x)}_{k}^{2} & = \sum_{i,j=1}^{\infty}\lambda_{i}\lambda_{j}D_{2}^{2}D_{1}^{2}k(x,x)[e_{i},e_{i},e_{j},e_{j}]\\
    & - 2\sum_{i=1}^{\infty}\lambda_{i}D_{2}^{2}D_{1}k(x,x)[x+CDU(x),e_{i},e_{i}]\\
    & + D_{2}D_{1}k(x,x)[x+CDU(x),x+CDU(x)],
\end{align*}
and using Assumption \ref{ass:nonvec_k_ass} there exists a constant $C_{k} > 0$ such that
\begin{align*}
     \norm{\xi(x)}_{k}^{2} & \leq  C_{k}\left(\sum_{i,j=1}^{\infty}\lambda_{i}\lambda_{j} + 2\sum_{i=1}^{\infty}\lambda_{i}\norm{x+CDU(x)}_{\calX}+ \norm{x+CDU(x)}_{\calX}^{2}\right)\\
     & = C_{k}\left(\norm{x+CDU(x)}_{\calX}+\text{Tr}[C]\right)^{2}.
\end{align*}
Then Assumption \ref{ass:P_Q_assumptions} assures us $\bbE_{Q}[\norm{\xi(X)}_{k}]<\infty$ and so
\begin{align*}
    \sup_{\substack{f\in\calH_{k},\\ \norm{f}_{k}\leq 1}}\bbE_{Q}[\calA f(X)] & =  \sup_{\substack{f\in\calH_{k},\\ \norm{f}_{k}\leq 1}}\bbE_{Q}[\langle f,\xi(X)\rangle_{k}] \\
    & =  \sup_{\substack{f\in\calH_{k},\\ \norm{f}_{k}\leq 1}}\langle f,\bbE_{Q}[\xi(X)]\rangle_{k}\\
    & = \norm{\bbE_{Q}[\xi(X)]}_{k},
\end{align*}
where the first equality is by the reproducing properties as outlined above, the second is by swapping expectation and inner product as is allowed for Bocher differentiable functions \citep[Theorem 3.1.7]{Hsing2015} and the final equality is by Cauchy-Schwarz. It does not matter that we haven't used the absolute value around the expectation in the first setting since $\calH_{k}$ is a vector space so the supremum is unchanged. The final stage of the proof is again swapping expectation and using the reproducing properties
\begin{align*}
    \norm{\bbE_{Q}[\xi(X)]}_{k}^{2} = \langle \bbE_{Q}[\xi(X)],\bbE_{Q}[\xi(X')]\rangle_{k} = \bbE_{Q\times Q}[\langle \xi(X),\xi(X')\rangle_{k}]
\end{align*}
where
\begin{align*}
    \langle \xi(x),\xi(x')\rangle_{k} & = \sum_{i,j=1}^{\infty}\lambda_{i}\lambda_{j}D_{2}^{2}D_{1}^{2}k(x,x')[e_{i},e_{i},e_{j},e_{j}]\\
    & - \sum_{i=1}^{\infty}\lambda_{i}D_{2}^{2}D_{1}k(x,x')[x+CDU(x),e_{i},e_{i}]\\
   &  - \sum_{i=1}^{\infty}\lambda_{i}D_{1}^{2}D_{2}k(x,x')[x'+CDU(x'),e_{i},e_{i}]\\
    & + D_{2}D_{1}k(x,x)[x+CDU(x),x+CDU(x)]\\
    & = (\calA\otimes\calA)k(x,x'),
\end{align*}
as required. 

The case for $\calAv$ is entirely analogous with 
\begin{align*}
    \xi_{v}(x) & = \sum_{i=1}^{\infty}e_{i}\lambda_{i}D_{1}k(x,\cdot)[e_{i}] - K(x,\cdot)[x+CDU(x)]\\
    & = \sum_{i=1}^{\infty}e_{i}\Gamma_{i}k(x,\cdot)\in\calH_{K}
\end{align*}
instead of $\xi(x)$.

\subsubsection{Proof of Propostion \ref{prop:ker_example}}
For the $k=k_{\text{SE-}T}$ case the first partial derivatives are $D_{1}k(x,y)[u]  = -D_{2}k(x,y)[u] = -\langle u,T^{*}T(x-y)\rangle_{\calX}k(x,y)$. Making the natural Riesz identification between $L(\calX,\bbR)$ and $\calX$ 
\begin{align*}
\norm{D_{1}k(x,y)}_{\calX} = \norm{D_{2}k(x,y)}_{\calX}\leq \norm{T}_{L(\calX)}\norm{T(x-y)}_{\calX}k(x,y)<c_{1},
\end{align*} 
for some $c_{1}\in\bbR$ not depending on $x,y$ since $te^{-t^{2}/2}$ is bounded for $t\in\bbR$. The second order mixed partial derivatives are
\begin{align*}
	D_{2}D_{1}k(x,y)[u,v] = (\langle T^{*}Tu,v\rangle_{\calX} + \langle T^{*}T(x-y),u\rangle_{\calX}\langle T^{*}T(x-y),v\rangle_{\calX})k(x,y),
\end{align*}
therefore
\begin{align*}
	\norm{D_{2}D_{1}k(x,y)}_{L(\calX\times\calX,\bbR)}\leq \norm{T}_{L(\calX)}^{2}(1+\norm{T(x-y)}_{\calX}^{2})k(x,y) < c_{2},
\end{align*}
for some $c_{2}<\infty$ not depending on $x,y$ since $te^{-t/2}$ is bounded for $t\geq 0$.  

For the $k = k_{\text{IMQ-}T}$ case, $D_{1}k(x,y)u = -\langle u,T^{*}T(x-y)\rangle_{\calX}k(x,y)^{3}$ so $$\norm{D_{1}k(x,y)}_{\calX}\leq \norm{T}_{L(\calX)}\norm{T(x-y)}_{\calX}k(x,y)^{3}<c_{3},$$ for some $c_{3}<\infty$ not depending on $x,y$ since $t(t^{2}+1)^{-3/2}$ is bounded. The $D_{2}$ derivative is handled similarly. The second order mixed partial derivatives are 
\begin{align*}
	D_{2}D_{1}k(x,y)[u,v] = \langle T^{*}Tu,v\rangle_{\calX}k(x,y)^{3} - 3\langle T^{*}T(x-y),u\rangle_{\calX}\langle T^{*}T(x-y),v\rangle_{\calX}k(x,y)^{5} ,
\end{align*}
therefore
\begin{align*}
	\norm{D_{2}D_{1}k(x,y)}_{L(\calX\times\calX,\bbR)}\leq \norm{T}_{L(\calX)}^{2}k(x,y)^{3} + 3\norm{T}^{2}_{L(\calX)}\norm{T(x-y)}_{\calX}^{2}k(x,y)^{5}<c_{4},
\end{align*}
for some $c_{4}<\infty$ independent of $x,y$ since $t^{2}(t^{2}+1)^{-5/2}$ is bounded. The calculations for $h_{v}$ in each of the two cases follow immediately from the calculations above of the partial derivatives.

\subsection{Proofs for Section \ref{sec:KSD_Tik}}
This section covers the proofs for Proposition \ref{prop:fourier_solution}, Proposition \ref{prop:unique}, Theorem \ref{thm:Tik_KSD}, Theorem \ref{thm:spec_separates} and Theorem \ref{corr:separates}. Multiple results are simple applications of results from papers in the literature, in such cases we have made an effort to translate our notation into the notation and language of the referenced results. 

\subsubsection{Proof of Proposition \ref{prop:fourier_solution}}
The proof is identical to \citet[Remark 3.13]{Bogachev1995regularity}. All that is required is that, in the notation of \citep{Bogachev1995regularity}, $\norm{B}_{\calX}\in L^{1}(\calX;Q)$ where, in our notation, $B(x) = -x-CDU(x)$ and so Assumption \ref{ass:P_Q_assumptions} takes care of this. 

\subsubsection{Proof of Proposition \ref{prop:unique}}
The fact that $P$ solves the equation is a simple consequence of an infinite-dimensional integration-by-parts formula obtained using logarithmic gradients \citep[Chapter 6]{Bogachev2010}. This is analogous to how it is shown that the target measure results in a zero KSD value in the existing literature, see for example \citep[Lemma 5.1]{Chwialkowski2016}. For more detail on the integration by parts method in infinite dimensions see Section \ref{subsec:deriv_measures} and the proof of Theorem \ref{thm:spec_separates}. 

The uniqueness result is the content of \citep[Theorem 4.5]{Albeverio1999uniqueness}. The assumptions of this result state that the base measure must satisfy a logarithmic-Sobolev inequality, be shift-equivalent along a dense subspace of $\calX$, possess a logarithmic gradient that is $L^{2}$ integrable with respect to the base measure and that the closure of the domain of the associated Dirichlet form must contain the square root of the density of $P$ with respect to the base measure. 

This is all satisfied in our scenario since our base measure is $N_{C}$, the centered Gaussian with non-degenerate covariance operator $C$. It is well known such measures enjoy the logarithmic-Sobolev inequality \citep[Theorem 10.30]{DaPrato2006}, are shift-equivalent along the corresponding Cameron-Martin space \citep[Theorem 2.8]{DaPrato2006} which is dense in $\calX$ \citep[Remark 1.3.2]{Maniglia2004}, the logarithmic gradient is simply $\beta^{N_{C}}_{Ch}(x) = -\langle h,x\rangle_{\calX}$ which immediately satisfies the integrability condition. Finally, the closure of the domain of the corresponding Dirichlet form is the Sobolev space $W^{1,2}_{C}(\calX)$ \citep[Proposition 1.2.3]{DaPrato2002}. Therefore Assumption \ref{ass:P_Q_assumptions} ensures that the square root of the density belongs to $W^{1,2}_{C}(\calX)$ along with the other required integrability conditions. 

Having established the conditions of \citet[Theorem 4.5]{Albeverio1999uniqueness} are satisfied the proof is complete as this result gives us the desired uniqueness condition.

\subsubsection{Proof of Theorem \ref{thm:Tik_KSD}}
This proof shall cover the case for the non-vectorised operator $\calA$ in \eqref{eq:KSD_nonvec_int} which is the more complex case as it involves more derivatives. The case for $\calAv$ in \eqref{eq:KSD_vec_int} follows using similar calculations. The method of proof is to expand \eqref{eq:KSD_nonvec_int} and arrive at \eqref{eq:nonvec_double_int}. 

Start by noting
\begin{align}
    \calA(e^{i\langle s,\cdot\rangle_{\calX}})(x) = ie^{i\langle s,x\rangle_{\calX}}\left(i\langle Cs,s\rangle_{\calX} - \langle x+CDU(x),s\rangle_{\calX}\right),\label{eq:A_exp}
\end{align}
so the complex conjugate is
\begin{align}
     \overline{\calA(e^{i\langle s,\cdot\rangle_{\calX}})(x)} = -ie^{-i\langle s,x\rangle_{\calX}}\left(-i\langle Cs,s\rangle_{\calX} - \langle x+CDU(x),s\rangle_{\calX}\right).\label{eq:conj_A_exp}
\end{align}
Using this, one can expand the integrand of \eqref{eq:KSD_nonvec_int} as
\begin{align*}
    \left|\bbE_{Q}[\calA (e^{i\langle s, \cdot\rangle_{\calX}})(X)] \right|_{\mathbb{C}}^2 & = \bbE_{Q}[\calA(e^{i\langle s,\cdot\rangle_{\calX}})(X)]\overline{\bbE_{Q}[\calA(e^{i\langle s,\cdot\rangle_{\calX}})(X')]}\\
    & = \int_{\calX}\int_{\calX} \calA(e^{i\langle s,\cdot\rangle_{\calX}})(x)\overline{\calA(e^{i\langle s,\cdot\rangle_{\calX}})(x')}dQ(x)dQ(x').
\end{align*}
Using \eqref{eq:A_exp} and \eqref{eq:conj_A_exp} this double integral is equal to
\begin{align}
    \int_{\calX}\int_{\calX}e^{i\langle s,x-x'\rangle_{\calX}}\langle Cs,s\rangle_{\calX}^{2} & + ie^{i\langle s,x-x'\rangle_{\calX}}\langle x+CDU(x),s\rangle_{\calX}\langle Cs,s\rangle_{\calX} \label{eq:expand_double}\\
    & - ie^{i\langle s,x-x'\rangle_{\calX}}\langle x'+CDU(x'),s\rangle_{\calX}\langle Cs,s\rangle_{\calX}\nonumber \\
    & + e^{i\langle s,x-x'\rangle_{\calX}}\langle x+CDU(x),s\rangle_{\calX}\langle x'+CDU(x'),s\rangle_{\calX}dQ(x)dQ(x').\nonumber
\end{align}

Recalling that \eqref{eq:KSD_nonvec_int} has the integral with respect to $\mu$ we now swap the order of integration so that we integrate with respect to $\mu$ first and then $Q$ twice afterwards. This is possible given the integrability assumptions of $U$ in Assumption \ref{ass:P_Q_assumptions} and the derivative assumptions on $k$ in Assumption \ref{ass:nonvec_k_ass} which translate to bounded moment assumptions on $\mu$. Bringing the integral with respect to $\mu$ on the inside means we can now identify terms in \eqref{eq:expand_double} with terms in \eqref{eq:nonvec_double_int}.

To make the comparison of terms consider first the expressions for the partial derivatives of the kernel
\begin{align*}
    k(x,x') & = \int_{\calX}e^{i\langle s,x-x'\rangle_{\calX}}d\mu(s)\\
    D_{1}D_{2}k(x,x')[u,v] & = \int_{\calX}\langle s,u\rangle_{\calX}\langle s,v\rangle_{\calX}e^{i\langle s,x-x'\rangle_{\calX}}d\mu(s)\\
    D_{1}^{2}D_{2}k(x,x')[u,v,w] & = \int_{\calX}ie^{i\langle s,x-x'\rangle_{\calX}}\langle s,u\rangle_{\calX}\langle s,v\rangle_{\calX}\langle s,w\rangle_{\calX}d\mu(s)\\
    D_{2}^{2}D_{1}k(x,x')[u,v,w] & = \int_{\calX}-ie^{i\langle s,x-x'\rangle_{\calX}}\langle s,u\rangle_{\calX}\langle s,v\rangle_{\calX}\langle s,w\rangle_{\calX}d\mu(s)\\
    D_{2}^{2}D_{1}^{2}k(x,x')[u,v,w,z] & = \int_{\calX}e^{i\langle s,x-x'\rangle_{\calX}}\langle s,u\rangle_{\calX}\langle s,v\rangle_{\calX}\langle s,w\rangle_{\calX}\langle s,z\rangle_{\calX}d\mu(s).
\end{align*}
Noting that $\langle Cs,s\rangle_{\calX} = \sum_{i=1}^{\infty}\lambda_{i}\langle s,e_{i}\rangle_{\calX}^{2}$ where $\lambda_{i},e_{i}$ is the eigensystem of $C$ we can equate the terms in \eqref{eq:nonvec_double_int} and \eqref{eq:expand_double}. Specifically,
\begin{align*}
    & \int_{\calX}e^{i\langle s,x-x'\rangle_{\calX}}\langle x+CDU(x),s\rangle_{\calX}\langle x'+CDU(x'),s\rangle_{\calX}d\mu(s) \\
    &\qquad\qquad = D_{1}D_{2}k(x,x')[x+CDU(x),x'+CDU(x')]\\
    & \int_{\calX}ie^{i\langle s,x-x'\rangle_{\calX}}\langle x+CDU(x),s\rangle_{\calX}\langle Cs,s\rangle_{\calX}d\mu(s) \\
    &\qquad\qquad = -\sum_{i=1}^{\infty}\lambda_{i}D_{2}^{2}D_{1}k(x,x')[x+CDU(x),e_{i},e_{i}]\\
    & \int_{\calX}-ie^{i\langle s,x-x'\rangle_{\calX}}\langle x'+CDU(x'),s\rangle_{\calX}\langle Cs,s\rangle_{\calX}d\mu(s) \\
    &\qquad\qquad = -\sum_{i=1}^{\infty}\lambda_{i}D_{1}^{2}D_{2}k(x,x')[x'+CDU(x'),e_{i},e_{i}]\\
    & \int_{\calX}e^{i\langle s,x-x'\rangle_{\calX}}\langle Cs,s\rangle_{\calX}^{2}d\mu(s) \\
    &\qquad\qquad = \sum_{i,j=1}^{\infty}\lambda_{i}\lambda_{j}D_{2}^{2}D_{1}k(x,x')[e_{i},e_{i},e_{j},e_{j}].
\end{align*}
This shows that each of the terms in \eqref{eq:expand_double} matches with terms in \eqref{eq:nonvec_double_int} which completes the proof.

\subsubsection{Proof of Theorem \ref{thm:spec_separates}}
Suppose that the KSD expressions are zero. Then the integrands in Theorem \ref{thm:Tik_KSD} must be zero outside of a $\mu$ measure zero set. Since $\mu$ has full support every open set that is non-empty has positive measure. As the two integrands are continuous functions of $s$ we can conclude that for all $s$ the integrands are zero. Then using Proposition \ref{prop:fourier_solution} we can conclude $\calA^{*}Q = 0$ and Proposition \ref{prop:unique} then implies $Q=P$.

On the other hand if $Q=P$ then Proposition \ref{prop:IBP} can be employed to conclude the KSD values are zero. The case for the non-vectorised operator $\calA$ is treated explicitly with the vectorised operator case being similar. Using Theorem \ref{thm:Tik_KSD} it is enough to show $\bbE_{P}[\calA(e^{i\langle s,\cdot\rangle_{\calX}})(X)] = 0 \:\forall s\in\calX$. For ease of notation set $g_{s} = e^{i\langle s,\cdot\rangle_{\calX}}$ then
\begin{align*}
    \bbE_{P}[\calA g_{s}(X)] & = \int_{\calX}\Tr[CD^{2}g_{s}(x)] - \langle x+CDU(x),Dg_{s}(x)\rangle_{\calX}dP(x)\\
    & = \sum_{i=1}^{\infty}\int_{\calX}\langle Ce_{i},D^{2}g_{s}(x)[e_{i}]\rangle_{\calX}-\langle x+CDU(x),e_{i}\rangle_{\calX}Dg_{s}(x)[e_{i}]dP(x),
\end{align*}
so employing Proposition \ref{prop:IBP} with $h = Ce_{i}, f = Dg_{s}(\cdot)[e_{i}]$ makes each term in this sum zero as required. We may employ Proposition \ref{prop:IBP} in this case since $Dg_{s}(x)[e_{i}] = i\langle s,e_{i}\rangle_{\calX}e^{i\langle s,x\rangle_{\calX}}$ is bounded as a function of $x$ for each $s$ and differentiable as a function of $x$. 

As mentioned the case for the KSD built using $\calAv$ is similar. Since $\Gamma$ returns $\calX$-valued functions one would need to check all the coefficients with respect to $\{e_{i}\}_{i=1}^{\infty}$ are zero. This is done using Proposition \ref{prop:IBP} with $h=Ce_{i}$ again and $f = g_{s}$.

\subsubsection{Proof of Theorem \ref{corr:separates}}
The proof strategy is similar to the strategy employed in \citet[Theorem 9]{Wynne2022}. The idea of the proof is to write KSD using $T$ as a limit of a KSD using a $T_{n}=I_{n}^{1/2}T$, where $T_{n}$ converges to $T$ since $I_{n}$ will converge to $I_{\calX}$. The key points is that for every $n$, $T_{n}$ satisfies Theorem \ref{thm:spec_separates}. Therefore, KSD using $T$ will be written as a limit of KSD expressions, each of which can separate measures. The final step is to ensure that the limit of these KSD expressions results in an expression which can still separate measures. This is done by providing an explicit upper bound on KSD using $T$ in the case where $Q\neq P$. The SE-$T$ kernel is dealt with first and then the IMQ-$T$ kernel as a corollary. For ease of notation denote the SE-$T$ kernel as $k_{T}$. 

To this end define $I_{n} = \sum_{i=1}^{\infty}\omega_{i}^{(n)}e_{i}\otimes e_{i}$ where $w_{i}^{(n)} = 1\: i\leq n$ and $i^{-2}\:\forall i > n$ and $e_{i}$ is an orthonormal basis to be specified later in this proof. Note that $I_{n}\in L^{+}_{1}(\calX)$ and is injective and approximates $I_{\calX}$ since as $n$ increases more of its eigenvalues become $1$. Denote by $h_{T}$ the Stein kernel obtained from the SE-$T$ kernel and the vectorised Stein operator. By Proposition \ref{prop:ker_example}
\begin{align*}
h_{T}(x,y)  & = k_{T}(x,y)\big(\langle x,y\rangle_{\calX} - \langle TC(x-y),T(x-y)\rangle_{\calX} \\
	& - \langle TC(CDU(x)-CDU(y)),T(x-y)\rangle_{\calX} + \text{Tr}(T^{*}TC^{2}) - \norm{CT^{*}T(x-y)}_{\calX}^{2}\big).
\end{align*}
Note this is the pointwise limit of $h_{I_{n}^{1/2}T}$. The convergence is immediate by pointwise convergence of $I_{n}$ to $I_{\calX}$, the only term which perhaps requires more justification is the trace term
\begin{align*}
\lvert\text{Tr}(T^{*}I_{n}TC)-\text{Tr}(T^{*}TC)\rvert&  = \lvert\text{Tr}(I_{n}T^{*}TC)-\text{Tr}(T^{*}TC)\rvert \\
&  = \sum_{i=1}^{\infty}\langle (I-I_{n})T^{*}CTe_{i},e_{i}\rangle_{\calX} \\
& = \sum_{i=n+1}^{\infty}(1-i^{-2})\langle T^{*}CTe_{i},e_{i}\rangle_{\calX}\leq \sum_{i=n+1}^{\infty}\langle T^{*}CTe_{i},e_{i}\rangle_{\calX} \rightarrow 0,
\end{align*}
since $T^{*}CT$ is trace class. 

With this pointwise convergence established
\begin{align}
    \text{KSD}_{\calAv,k_{T}}(Q,P)^{2} & = \int_{\calX}\int_{\calX}h_{T}(x,x')dQ(x)dQ(x')\label{eq:double_int_rep}\\
    & = \int_{\calX}\int_{\calX}\lim_{n\rightarrow\infty}h_{I_{n}^{1/2}T}(x,x')dQ(x)dQ(x')\label{eq:double_int_lim}\\
    & = \lim_{n\rightarrow\infty}\int_{\calX}\int_{\calX}h_{I_{n}^{1/2}T}(x,x')dQ(x)dQ(x')\label{eq:DCT}\\
    & = \lim_{n\rightarrow\infty}\text{KSD}_{k_{I_{n}^{1/2}T}}(Q,P)^{2},\label{eq:KSD_I_n}
\end{align}
where $K_{T} = k_{T}I_{\calX}$, \eqref{eq:double_int_rep} is Corollary \ref{cor:stein_kernels}, \eqref{eq:double_int_lim} is the pointwise convergence of the Stein kernel, \eqref{eq:DCT} is the dominated convergence theorem which applies by the integrability assumptions made and \eqref{eq:KSD_I_n} is again Corollary \ref{cor:stein_kernels}.

So far we have shown the KSD using SE-$T$ kernel may be written as the limit of the KSD using the SE-$I_{n}^{1/2}T$ kernel. Note that the SE-$I_{n}^{1/2}T$ kernel is equal to $\widehat{N}_{T^{*}I_{n}T}$. This means we can employ previous results which held for kernels that are Fourier transforms of measures. Namely, by Theorem \ref{thm:Tik_KSD}
\begin{align}
    \text{KSD}_{\calA_v, K_{T}}(Q,P)^{2} & =\lim_{n\rightarrow\infty} \int_{\calX}\left\|\mathbb{E}_Q\left[\Gamma(e^{i\left<s,\cdot\right>_{\calX}})(X)\right]\right\|_{\calX_{\mathbb{C}}}^2 dN_{T^{*}I_{n}T}(s)\label{eq:lim_KSD}\\
    & = \lim_{n\rightarrow\infty} \int_{\calX}\left\|\mathbb{E}_Q\left[\Gamma(e^{i\left<T^{*}s,\cdot\right>_{\calX}})(X)\right]\right\|_{\calX_{\mathbb{C}}}^2 dN_{I_{n}}(s),\label{eq:lim_KSD_T}
\end{align}
where \eqref{eq:lim_KSD_T} is the Gaussian change of variable formula \citep[Proposition 1.1.8]{DaPrato2006}. If $P=Q$ then by the same argument of Theorem \ref{thm:spec_separates} the integrand in \eqref{eq:lim_KSD} is zero for every $n$ and so the KSD is the limit of a sequence whose every value is zero and thus is zero. This completes the proof in one direction.

Now suppose $Q\neq P$. For ease of notation set 
\begin{align*}
    F(s) & \coloneqq \left\|\mathbb{E}_Q\left[\Gamma(e^{i\left<s,\cdot\right>_{\calX}})(X)\right]\right\|_{\calX_{\mathbb{C}}}^2\\
    & = \left\|Cs\widehat{Q}(s)+D\widehat{Q}(s)+i\int_{\calX}CDU(x)e^{i\langle s,x\rangle_{\calX}}dQ(x)\right\|_{\calX_{\bbC}}^{2}.
\end{align*}

The idea for this direction of the proof is to lower bound KSD using $T$ by something positive. This will be done again by using the limiting argument. The idea will be to find a point where the integrand in the Fourier representation is positive. A result by \citet[Remark 3.13]{Bogachev1995regularity} states $\mathbb{E}_Q\left[\Gamma(e^{i\left<s,\cdot\right>_{\calX}})(X)\right] = 0\:\forall s\in\calX$ implies $Q=P$. Using the contrapositive means there exists some $s_{0}\in\calX$ such that $ \mathbb{E}_Q\left[\Gamma(e^{i\left<s_{0},\cdot\right>_{\calX}})(X)\right]\neq 0$ hence $F(s_{0}) > 0$. Since $T^{*}$ is surjective there exists some $u_{0}\in\calX$ such that $T^{*}u_{0} = s_{0}$ hence $F(T^{*}u_{0}) > 0$. Now that we have found a single point where the integrand is positive, the idea is to find a large set in $\calX$ on which $F(T^{*}\cdot)$ remains positive. This set should be large enough so that it has measure with respect to all of the $N_{I_{n}}$ that is bounded below by a positive constant. The way to find this large set is to deduce that the integrand is very slow varying, so that if it is positive at $s_{0}$ it has to be positive in a big set around $s_{0}$. The next result makes this notion concrete.

\begin{proposition}\label{prop:V_continuity}
Suppose $\bbE_{Q}[\norm{X}_{\calX}^{2}]<\infty$ and $\bbE_{Q}[\norm{CDU(X)}_{\calX}^{2}]<\infty$ then there exists an injective $V\in L^{+}_{1}(\calX)$ such that $F$ is continuous with respect to the norm $\norm{\cdot}_{V} = \langle V\cdot,\cdot\rangle_{\calX}^{1/2}$. 
\end{proposition}

\begin{proof}
By definition $C$ is continuous with respect to $\norm{\cdot}_{C^{2}}$ and by the Minlos-Sazonov theorem, see \citet[Theorem VI.1.1]{Vakhania1987}, there exists some $U\in L^{+}_{1}(\calX)$ such that $\widehat{Q}$ is continuous with respect to $\norm{\cdot}_{U}$. Recall that $D\widehat{Q}(s) = \int_{\calX} ixe^{i\langle x,s\rangle}dQ(x)$ so
\begin{align}
    \norm[\big]{D\widehat{Q}(s)-D\widehat{Q}(t)}_{\calX_{\bbC}}\leq \int_{\calX} \norm[\big]{x(e^{i\langle x,s\rangle_{\calX}}-e^{i\langle x,t\rangle_{\calX}})}_{\calX_{\bbC}}dQ(x).\label{eq:D_bound}
\end{align}
The square of the integrand is
\begin{align*}
    \norm[\big]{x(e^{i\langle x,s\rangle_{\calX}}-e^{i\langle x,t\rangle_{\calX}})}_{\calX_{\bbC}}^{2} = 2\norm{x}_{\calX}^{2}(1-\cos(\langle x,s-t\rangle_{\calX})),
\end{align*}
since
\begin{align*}
    \text{Re}\left(x(e^{i\langle x,s\rangle_{\calX}}-e^{i\langle x,t\rangle_{\calX}})\right) & = x(\cos(\langle x,s\rangle_{\calX})-\cos(\langle x,t\rangle_{\calX}))\\
    \text{Im}\left(x(e^{i\langle x,s\rangle_{\calX}}-e^{i\langle x,t\rangle_{\calX}})\right) & = x(\sin(\langle x,s\rangle_{\calX})-\sin(\langle x,t\rangle_{\calX})).
\end{align*}
Substituting into \eqref{eq:D_bound} 
\begin{align*}
    \norm[\big]{D\widehat{Q}(s)-D\widehat{Q}(t)}_{\calX_{\bbC}} & \leq \int_{\calX} \norm{x}_{\calX}2^{1/2}(1-\cos(\langle x,s-t\rangle_{\calX}))^{1/2}dQ(x) \\
    & \leq \int_{\calX} \norm{x}_{\calX}\lvert\langle x,s-t\rangle_{\calX}\rvert dQ(x)\\
    & \leq \bbE_{Q}[\norm{X}_{\calX}^{2}]^{1/2}\bbE_{Q}[\langle x,s-t\rangle_{\calX}^{2}]^{1/2}\\
    & \leq c_{Q}\langle W(s-t),s-t\rangle_{\calX}^{1/2},
\end{align*}
for some $W\in L^{+}_{1}(\calX)$ the covariance operator associated with $Q$ which exists due to the finite second moment assumption \citep[Lemma 1.1.4]{Maniglia2004}, $c_{Q}$ is some finite constant from this second moment assumption and we used the standard inequality $2(1-\cos(x))\leq x^{2},x\in\bbR$. This shows that $D\widehat{Q}$ is continuous, Lipschitz continuous in fact, with respect to the norm induced by $W$. 

The same exact argument can be made to show that $i\int_{\calX}CDU(x)e^{i\langle\cdot,x\rangle_{\calX}}dQ(x)$ is also Lipschitz continuous with respect to some $R\in L^{+}_{1}(\calX)$ by simply replacing $x$ in the integrand with $CDU(x)$ in the above derivation. Therefore $\norm{C(\cdot)\widehat{Q}(\cdot)+D\widehat{Q}(\cdot)+i\int_{\calX}CDU(x)e^{i\langle\cdot,x\rangle_{\calX}}dQ(x)}_{\calX_{\bbC}}^{2}$ is continuous with respect to the norm induced by $C^{2} + U + W+R\in L^{+}_{1}(\calX)$ from which we can obtain an injective element of $L^{+}_{1}(\calX)$ by modifying the eigenvalues to all be positive. 
\end{proof}

By Proposition \ref{prop:V_continuity} $F$ is continuous with respect to $\norm{\cdot}_{V}$ hence $F(T^{*}\cdot)$ is continuous with respect to $U\coloneqq TVT^{*} L^{+}_{1}(\calX)$ and we can form a $\tilde{U}$ from $U$ that is injective by making any zero eigenvalues non-zero. Since this $\tilde{U}$ would result in larger norm values $F(T^{*}\cdot)$ is also continuous with respect to $\norm{\cdot}_{\tilde{U}}$. Forming an injective operator from a possibly non-injective one is done so that the Gaussian measure which has covariance operator the same as the operator used in the continuity statement has full support. Now that we know $F(T^{*}\cdot)$ has this strong continuity property, and that it is positive at a point, the next result shows that the limit of integrals with respect to $N_{I_{n}}$ is positive. 

\begin{theorem}\label{thm:lower_bound}
Let $V\in L^{+}_{1}(\calX)$ be injective. Suppose $f\colon\calX\rightarrow [0,\infty)$ is continuous with respect to the norm $\norm{\cdot}_{V}= \langle V \cdot,\cdot\rangle_{\calX}^{1/2}$ and define $I_{n} = \sum_{i=1}^{\infty}\omega_{i}^{(n)}e_{i}\otimes e_{i}$ where $\omega_{i}^{(n)} = 1$ if $i\leq n$ and $\omega_{i}^{(n)} = i^{-2}$ if $i > n$ and $\{e_{i}\}_{i=1}^{\infty}$ is the eigenbasis of $V$. If there exists a point $x_{0}\in\calX$ such that $f(x_{0}) > 0$ then 
\begin{align*}
    \lim_{n\rightarrow\infty}\int_{\calX}f(x)dN_{I_{n}}(x) > 0.
\end{align*}
\end{theorem}

\begin{proof}
This proof is largely formalising the intuition that if $f$ is slow varying and positive at a point then it must be positive in a big set, namely a ball with respect to a norm induced by a trace class operator. Then even as the $N_{I_{n}}$ measures contract as $n$ gets larger, the measure of the set is so large that it is bounded below. Let $\varepsilon > 0$ be such that $f(x_{0}) > \varepsilon$ then since $f$ is continuous with respect to $\norm{\cdot}_{V}$ there exists an $r > 0$ such that $f(x) > \varepsilon/2$ for all $x\in B_{V}(x_{0},r)$, the ball based at $x_{0}$ with radius $r$ with respect to $\norm{\cdot}_{V}$. Therefore
\begin{align*}
    \lim_{n\rightarrow\infty}\int_{\calX}f(x)dN_{I_{n}}(x) \geq \lim_{n\rightarrow\infty}\int_{B_{V}(x_{0},r)}\frac{\varepsilon}{2} dN_{I_{n}}(x),
\end{align*}
so it is sufficient to show the limit of the measure of this set is positive. 
Using standard change of variable formulas
\begin{align*}
    N_{I_{n}}(B_{V}(x_{0},r)) = N_{-x_{0},I_{n}}(B_{V}(0,r)) = N_{-V^{1/2}x_{0},V^{1/2}I_{n}V^{1/2}}(B(0,r)),
\end{align*}
and for ease of notation let $y_{0} = -V^{1/2}x_{0}$ and note $V_{n}\coloneqq V^{1/2}I_{n}V^{1/2} = \sum_{i=1}^{\infty}\lambda_{i}^{(n)}e_{i}\otimes e_{i}$ where $\lambda_{i}^{(n)} = \lambda_{i}$ for $i\leq n$ and $\lambda_{i}i^{-2}$ for $i > n$ where $\lambda_{i}$ are the eigenvalues of $V$ which are strictly positive by the assumption that $V$ is injective. 

Now we use the proof technique of \citet[Proposition 1.25]{DaPrato2006}. Set \begin{align*}
    A_{l}  & = \{x\in\calX\colon\sum_{i=1}^{l}\langle x,e_{i}\rangle_{\calX}^{2}\leq r^{2}/2\}\\
     B_{l}  & = \{x\in\calX\colon\sum_{i=l+1}^{\infty}\langle x,e_{i}\rangle_{\calX}^{2}\leq r^{2}/2\},
\end{align*}
    meaning $A_{l},B_{l}$ are independent under $V_{n}$ since they depend on different parts of the eigenbasis. Therefore for every $l\in\bbN$, $N_{y_{0},V_{n}}(B(0,r))\geq N_{y_{0},V_{n}}(A_{l})N_{y_{0},V_{n}}(B_{l})$. The Markov inequality yields
\begin{align*}
    N_{y_{0},V_{n}}(B_{l})  = 1 - N_{y_{0},V_{n}}(B_{l}^{c})
    & \geq 1 - \frac{2}{r^{2}}\int_{\calX}\sum_{i=l+1}^{\infty}\langle x,e_{i}\rangle_{\calX}^{2}dN_{y_{0},V_{n}}(x)\\
    & = 1 - \frac{2}{r^{2}}\left(\sum_{i=l+1}^{\infty}\lambda_{i}^{(n)} + \langle y_{0},e_{i}\rangle_{\calX}^{2}\right)\\
    & \geq 1 -\frac{2}{r^{2}}\left(\sum_{i=l+1}^{\infty}\lambda_{i} + \langle y_{0},e_{i}\rangle_{\calX}^{2}\right),
\end{align*}
where the second inequality is by the definition of $\lambda_{i}^{(n)}$. Note that the final expression doesn't depend on $n$ and is the tail of a finite sum. Therefore there exists an $L$, independent of $n$, such that for $l\geq L$ we have $N_{y_{0},V_{n}}(B_{l}) > 1/2$. 

Now take any $n\geq L$ then $N_{y_{0},V_{n}}(A_{L}) = N_{y_{0},V}(A_{L})$ since $A_{L}$ only depends on the first $L$ components of $V_{n}$ which match the first $n$ components of $V$. Since $V$ is injective $N_{y_{0},V}(A_{L}) > c > 0$ for some $c$ \citep{DaPrato2006}[Proposition 1.25]. 

In conclusion, for every $n\geq L$, $N_{y_{0},V_{n}}(A_{L})N_{y_{0},V_{n}}(B_{L}) > c/2 > 0$ so
\begin{align*}
    \lim_{n\rightarrow\infty}\int_{\calX}f(x)dN_{I_{n}}(x) > c\varepsilon/4 > 0,
\end{align*}
and the proof is complete. 
\end{proof}

Setting $F(T^{*}\cdot)$ as $f$ and $\tilde{U}$ as $V$ in Theorem \ref{thm:lower_bound} completes the proof for the SE-$T$ case since by \eqref{eq:lim_KSD_T} we can conclude that $\text{KSD}_{\calAv,K_{T}}(Q,P) > 0$.

For the IMQ-$T$ case we begin using the same limiting argument and the result from Example \ref{exp:ker_spec} which shows how the Fourier measure of the IMQ kernel may be written in terms of a Gaussian measures on $\calX$ and a Gaussian measure on $\bbR$. Using $k_{T}$ to denote the SE-$T$ kernel and $h_{T}$ the Stein kernel corresponding to the IMQ-$T$ kernel 
\begin{align}
	\text{KSD}_{\calAv,k_{T}}(Q,P)^{2} & = \int_{\calX}\int_{\calX}h_{T}(x,y)dQ(x)dQ(y) \nonumber\\
	& = \lim_{n\rightarrow\infty}\int_{\calX}\int_{\calX}h_{I_{n}^{1/2}T}(x,y)dQ(x)dQ(y) \label{eq:lim}\\
	& = \lim_{n\rightarrow\infty}\int_{\bbR}\int_{\calX}\left\|F(\eta T^{*}s)\right\|_{\calX_{\bbC}}^{2}dN_{I_{n}}(s)dN_{1}(\eta)\label{eq:imq_rep}\\
	& = \int_{\bbR}\lim_{n\rightarrow\infty}\int_{\calX}\left\|F(\eta T^{*}s)\right\|_{\calX_{\bbC}}^{2}dN_{I_{n}}(s)dN_{1}(\eta)\label{eq:swap_lim}\\
	& = \int_{\bbR}\text{KSD}_{\calAv,k_{\text{SE-}\eta T}}(Q,P)^{2}dN_{1}(\eta),\label{eq:IMQ_low}
\end{align}
where \eqref{eq:lim} is the same limiting argument as \eqref{eq:DCT}, \eqref{eq:imq_rep} is Theorem \ref{thm:Tik_KSD} and the change of variables used in \eqref{eq:lim_KSD_T}, \eqref{eq:swap_lim} is the  dominated convergence theorem and \eqref{eq:IMQ_low} is by \eqref{eq:lim_KSD_T}. 

Note that for all $\eta\neq 0$ the SE-$\eta T$ kernel satisfies the assumptions required for the KSD based on the SE-$\eta T$ kernel to be separating. Therefore if $P=Q$ then \eqref{eq:IMQ_low} shows the KSD based on the IMQ-$T$ kernel is an integral of zero-valued functions and hence is zero. On the other hand is $P\neq Q$ then the KSD based on the IMQ-$T$ kernel is an integral of an almost everywhere positive function, hence is positive. This completes the proof.



\end{document}